\documentclass[12pt,a4paper,dvipsnames]{amsart}

\usepackage{algorithm,cite,makecell}
\usepackage{algpseudocode}
\usepackage{amsmath, nicefrac, amsthm, verbatim, amsfonts, amssymb, xcolor, bbm}
\usepackage{graphics, xspace, enumerate}
\usepackage[a4paper,margin=3cm]{geometry}
\usepackage{multirow}
\usepackage{setspace}
\usepackage{kbordermatrix}
\usepackage{multicol}

\usepackage[bb=boondox]{mathalfa}

\usepackage{float}

\usepackage{graphicx}
\usepackage[colorlinks=true,citecolor=green,urlcolor=green,linkcolor=green,bookmarksopen=true,unicode=true,pdffitwindow=true]{hyperref}
\usepackage[english]{babel}
\usepackage[languagenames,fixlanguage]{babelbib}

\usepackage[labelformat=simple]{subcaption}

\usepackage{tikz}
\usetikzlibrary{graphs,quotes,arrows.meta,bending}

\makeatletter
\@namedef{subjclassname@2020}{\textup{2020} Mathematics Subject Classification}
\makeatother

\theoremstyle{plain}
\newtheorem{theorem}{Theorem}[section]

\newtheorem{lemma}[theorem]{Lemma}
\newtheorem{proposition}[theorem]{Proposition}
\theoremstyle{definition}
\newtheorem{remark}[theorem]{Remark}

\newtheorem{definition}[theorem]{Definition}

\newcommand{\FLOOR}[1]{\left\lfloor #1 \right\rfloor}

\newcommand{\CEIL}[1]{\left\lceil #1 \right\rceil}

\newcommand{\FJADEF}[3]{{#1}:{#2}\to{#3}}
\newcommand{\Mod}[1]{\ (\mathrm{mod}\ #1)}

\numberwithin{equation}{section}

\definecolor{Maroon}{RGB}{140,10,0}

\hypersetup{
	colorlinks,
	linkcolor={Maroon},
	citecolor={Maroon},
	urlcolor={Maroon}
}

\title{Predators and altruists arriving on jammed Riviera}

\author[T.\ Do\v{s}li\'{c}]{Tomislav\ Do\v{s}li\'{c}}
\address[Tomislav\ Do\v{s}li\'{c}]{Department of Mathematics\\
	Faculty of Civil Engineering\\
	University of Zagreb\\
	Zagreb\\
	Croatia \\ and
Faculty of Information Studies \\
Novo Mesto \\
Slovenia}
\email{tomislav.doslic@grad.unizg.hr}

\author[M.\ Puljiz]{Mate\ Puljiz}
\address[Mate Puljiz]{Department of Applied Mathematics\\
	Faculty of Electrical Engineering and Computing\\
	University of Zagreb\\
 Zagreb\\
	Croatia}
\email{mate.puljiz@fer.hr}

\author[S.\ \v{S}ebek]{Stjepan\ \v{S}ebek}
\address[Stjepan\ \v{S}ebek]{Department of Applied Mathematics\\
	Faculty of Electrical Engineering and Computing\\
	University of Zagreb\\
 Zagreb\\
	Croatia}
\email{stjepan.sebek@fer.hr}

\author[J.\ \v{Z}ubrini\'{c}]{Josip\ \v{Z}ubrini\'{c}}
\address[Josip\ \v{Z}ubrini\'{c}]{Department of Applied Mathematics\\
	Faculty of Electrical Engineering and Computing\\
	University of Zagreb\\
 Zagreb\\
	Croatia}
\email{josip.zubrinic@fer.hr}



\begin{document}

\begin{abstract}
    The Riviera model is a combinatorial model for a settlement along a coastline, introduced recently by the authors. Of most interest are the so-called jammed states, where no more houses can be built without violating the condition that every house needs to have free space to at least one of its sides. In this paper, we introduce new agents (predators and altruists) that want to build houses once the settlement is already in the jammed state. Their behavior is governed by a different set of rules, and this allows them to build new houses even though the settlement is jammed. Our main focus is to detect jammed configurations that are resistant to predators, to altruists, and to both predators and altruists. We provide bivariate generating functions, and complexity functions (configurational entropies) for such jammed configurations. We also discuss this problem in the two-dimensional setting of a combinatorial settlement planning model that was also recently introduced by the authors, and of which the Riviera model is just a special case.
\end{abstract}

\maketitle

%
%
%
%

\section{Introduction}
In this paper, we expand the Riviera model introduced by the authors in \cite{DPSZ}. The Riviera model is a one-dimensional variant of a two-dimensional irreversible deposition model introduced in \cite{PSZ-21, PSZ-21-2}. In the original two-dimensional model, a rectangular $m \times n$ tract of land is considered. The sides of that tract of land are oriented north-south and east-west, and it consists of $mn$ square lots of size $1 \times 1$ (see Figure \ref{fig:tract_of_land}).
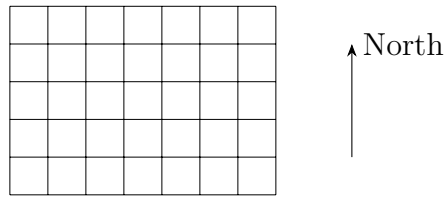
\begin{figure}[h]
	\centering
	\begin{tikzpicture}[scale = 0.5]
	\draw[step=1cm,black,very thin] (0, 0) grid (7,5);
	\draw [-Stealth] (9,1) -- (9,4);
	\node[anchor=west] at (9,4) {North};
	\end{tikzpicture}
	\caption{An example of a tract of land ($m = 5$, $n = 7$).}
	\label{fig:tract_of_land}
\end{figure}
Each $1 \times 1$ square lot can be either empty, or occupied by a single house. A house is said to be blocked from sunlight if the three lots immediately to its east, west and south are all occupied (it is assumed that sunlight always comes from the south). For the tracts of land along the eastern, western, and southern boundary of the rectangular $m \times n$ grid, there are no obstructions to sunlight. We refer to the models of such rectangular tracts of land, with certain lots occupied, as configurations. Of interest are the maximal (also referred to as jammed) configurations, where no house is blocked from the sunlight, and any further addition of a house to the configuration on any empty lot would result in either that house being blocked from the sunlight, or it would cut off sunlight from some previously built house, or both.

We encode any fixed configuration as a $0-1$, $m \times n$ matrix $C$, with $C_{i,j}=1$ if and only if a house is built on the lot $(i,j)$ ($i$-th row and $j$-th column, counted from the top left corner). It is natural to define the \emph{building density} $\rho$ of a configuration $C$ as $\rho = \dfrac{|C|}{mn}$, where
$$|C| = \sum_{i=1}^m\sum_{j=1}^n C_{i,j}$$
is the total number of occupied lots in the configuration $C$. We also refer to $|C|$ as the \emph{occupancy} of $C$.

A configuration $C$ is said to be \emph{permissible} if no house in it is blocked from the sunlight, otherwise it is called \emph{impermissible}.

A configuration $C$ is said to be \emph{maximal} (\emph{jammed}) if it is permissible and no other permissible configuration strictly contains it, i.e.\ no further houses can be added to it, whilst ensuring that all the houses still get some sunlight. See Figure \ref{fig:examples} for examples of impermissible, permissible, and maximal configurations on a $5 \times 4$ tract of land. In this figure (and all the following figures) shaded squares represent houses and unshaded squares represent empty lots on the tract of land.

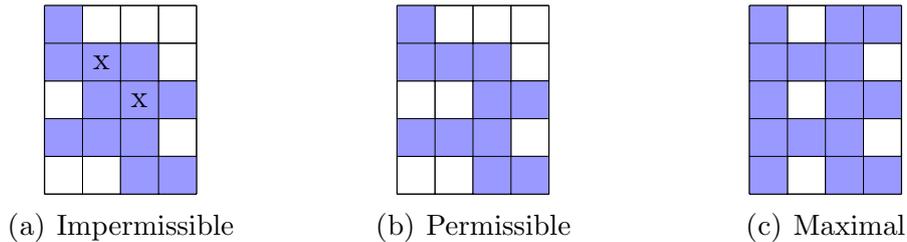
\begin{figure}
	\begin{subfigure}{0.3\textwidth}\centering
		\begin{tikzpicture}[scale = 0.5]
		\draw[step=1cm,black,very thin] (0, 0) grid (4,5);
		\fill[blue!40!white] (0,1) rectangle (1,2);
		\fill[blue!40!white] (0,3) rectangle (1,4);
		\fill[blue!40!white] (0,4) rectangle (1,5);
		\fill[blue!40!white] (1,1) rectangle (2,2);
		\fill[blue!40!white] (1,2) rectangle (2,3);
		\fill[blue!40!white] (1,3) rectangle (2,4);
		\node[] at (1.5,3.5) {x};
		\fill[blue!40!white] (2,0) rectangle (3,1);
		\fill[blue!40!white] (2,1) rectangle (3,2);
		\fill[blue!40!white] (2,2) rectangle (3,3);
		\node[] at (2.5,2.5) {x};
		\fill[blue!40!white] (2,3) rectangle (3,4);
		\fill[blue!40!white] (3,0) rectangle (4,1);
		\fill[blue!40!white] (3,2) rectangle (4,3);
		\draw[step=1cm,black,very thin] (0, 0) grid (4,5);
		\end{tikzpicture}
		\caption{Impermissible}
	\end{subfigure}
	\begin{subfigure}{0.3\textwidth}\centering
		\begin{tikzpicture}[scale = 0.5]
		\draw[step=1cm,black,very thin] (0, 0) grid (4,5);
		\fill[blue!40!white] (0,1) rectangle (1,2);
		\fill[blue!40!white] (0,3) rectangle (1,4);
		\fill[blue!40!white] (0,4) rectangle (1,5);
		\fill[blue!40!white] (1,1) rectangle (2,2);
		\fill[blue!40!white] (1,3) rectangle (2,4);
		\fill[blue!40!white] (2,0) rectangle (3,1);
		\fill[blue!40!white] (2,1) rectangle (3,2);
		\fill[blue!40!white] (2,2) rectangle (3,3);
		\fill[blue!40!white] (2,3) rectangle (3,4);
		\fill[blue!40!white] (3,0) rectangle (4,1);
		\fill[blue!40!white] (3,2) rectangle (4,3);
		\draw[step=1cm,black,very thin] (0, 0) grid (4,5);
		\end{tikzpicture}
		\caption{Permissible}
	\end{subfigure}
	\begin{subfigure}{0.3\textwidth}\centering
		\begin{tikzpicture}[scale = 0.5]
		\draw[step=1cm,black,very thin] (0, 0) grid (4,5);
		\fill[blue!40!white] (0,0) rectangle (1,1);
		\fill[blue!40!white] (0,1) rectangle (1,2);
		\fill[blue!40!white] (0,2) rectangle (1,3);
		\fill[blue!40!white] (0,3) rectangle (1,4);
		\fill[blue!40!white] (0,4) rectangle (1,5);
		\fill[blue!40!white] (1,1) rectangle (2,2);
		\fill[blue!40!white] (1,3) rectangle (2,4);
		\fill[blue!40!white] (2,0) rectangle (3,1);
		\fill[blue!40!white] (2,1) rectangle (3,2);
		\fill[blue!40!white] (2,2) rectangle (3,3);
		\fill[blue!40!white] (2,3) rectangle (3,4);
		\fill[blue!40!white] (2,4) rectangle (3,5);
		\fill[blue!40!white] (3,0) rectangle (4,1);
		\fill[blue!40!white] (3,2) rectangle (4,3);
		\fill[blue!40!white] (3,4) rectangle (4,5);
		\draw[step=1cm,black,very thin] (0, 0) grid (4,5);
		\end{tikzpicture}
		\caption{Maximal}
	\end{subfigure}
	\caption{Examples of impermissible, permissible and maximal configuration on a $5 \times 4$ tract of land. The houses that are blocked from the sunlight are marked with `x'.}\label{fig:examples}
\end{figure}

We were introduced to this problem by Juraj Bo\v{z}i\'{c} who came up with it during his studies at the Faculty of Architecture, University of Zagreb. His main goal was to design a model for settlement planning where the impact of architects, urbanists, and other regulators would be as small as possible, and people would have a lot of freedom in the process of building the settlement. This minimal intervention from the side of the regulator is given through the condition that houses are not allowed to be blocked from the sunlight, and that the tracts of land on which the settlements are built are of rectangular shapes.

The Riviera model is a one-dimensional modification of the settlement planning model described above, which
ignores the possibility of obtaining sunlight from the south, but instead retains only the constraints pertaining to the east and west directions. As this is a model on a strip of land, it resembles a Mediterranean settlement along the coast (riviera), hence the name. The configuration of built houses is represented with a row vector $C = (c_k)$, where $c_k = 1$ if the lot $k$ is occupied and $c_k = 0$ otherwise. We write configurations as strings of $0$'s and $1$'s, and we refer to any consecutive sequence of letters in a configuration as a substring or a (sub)word in that configuration. Similarly as before, a configuration is said to be permissible if every occupied lot has at least one neighboring lot unoccupied (except maybe for the first and the last lot which receive sunlight from the boundary) so that it is not blocked from the sunlight. Among permissible configurations, we are interested in the jammed ones, namely configurations such that any addition of a house on an unoccupied lot would result in an impermissible configuration.

When it comes to the set of all jammed configurations (for both the original two-dimensional settlement model, and the Riviera model), there are several natural questions that one can try to answer. Some of them are: How many different jammed configurations are there on an $m \times n$ tract of land ($m, n \in \mathbb{N}$)? How many of them have a particular building density? What is the minimal ($\rho_{\min}$), and what is the maximal ($\rho_{\max}$) density of a jammed configuration? What is the average density of a jammed configuration?

Some of these questions have been tackled in the general two-dimensional case (see \cite{PSZ-21, PSZ-21-2}), but due to the complexity of the general case, most of the exact results are known only for the Riviera model. To answer the question about the number of different jammed configurations in the Riviera model with prescribed length and occupancy, one can compute the bivariate generating function enumerating all such configurations, and this has been done in \cite{DPSZ}, and later reconstructed in \cite{krapivsky2022jamming}. It is very easy to see that for the Riviera model $\rho_{\min} = 1/2$, and $\rho_{\max} = 2/3$, but much more interesting question is the one about the average density. By average density we mean the expected value of the random variable measuring the density of a randomly sampled jammed configuration. When it comes to this question, we first have to clarify how we randomly sample a jammed configuration. Notice that we can do that in (at least) two natural ways. We refer to those two cases as the dynamic, and the equilibrium (static) version of the model. In the dynamic version of the model, a jammed configurations is reached by sequentially (and randomly) building houses until no more houses can be built without violating the permissibility condition. These kinds of models, where particles (houses \cite{PSZ-21-2}, atoms \cite{Krapiv20}, molecules \cite{gonzalez1974cooperative}, cars \cite{Page}, or some other type of particles) are randomly and sequentially introduced in a system, have been studied extensively in the literature, under the common name of Random sequential adsorption (RSA), and they have numerous applications in physics, chemistry and biology \cite{talbot2000car}. If a jammed configuration is randomly sampled in this way (by randomly and sequentially building houses until we reach one), the limit of average (expected) densities, as the length of the tract of land goes to infinity, is called the \emph{jamming limit}. In the equilibrium version of the model, we consider the set of all jammed configurations, and we sample uniformly at random one of them. The expected value of the density of a jammed configuration sampled in such a way converges (as the length of the configuration tends to infinity) to the argument of the maximum of the so-called \emph{complexity function} (also known as \emph{configurational entropy}) of the ensemble of all jammed configurations. We provide a precise definition of the complexity function in Section \ref{sec:MCRP}. After the Riviera model was introduced in \cite{DPSZ}, it immediately received attention in physics community (see \cite{krapivsky2022jamming}) due to its resemblance to standard RSA models that have already been extensively studied. Both jamming limit and configurational entropy of Riviera model have been computed in \cite{krapivsky2022jamming}. Even though one can sometimes calculate the jamming limit by analytical means, lack of the so-called shielding property (which turns out to be crucial) makes it nearly impossible to do it for the Riviera model. In such cases, the jamming limit is approximated by using Monte Carlo simulations, and this was exactly the approach used by the authors in \cite{krapivsky2022jamming}.

Our main focus in this paper will be the static model, but not in the original setting. We upgrade the Riviera model with two new categories of agents, namely predators and altruists. After a jammed configuration is reached, no more houses can be built without making this configuration impermissible. However, behavior of predators and altruist is guided by a philosophy that allows new houses to be built. Predators do not care about others, but they do care about themselves. This means that they will build a house on a an empty lot as long as this house will receive sunlight, regardless of the fact that they could block sunlight to some other house (see Figure \ref{fig:Predator_arriving}).
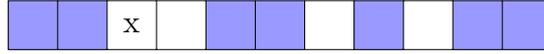
\begin{figure}
	\centering
	\begin{tikzpicture}[scale = 0.65]
		\fill[blue!40!white] (0,0) rectangle (1,1);
            \fill[blue!40!white] (1,0) rectangle (2,1);
            \fill[blue!40!white] (4,0) rectangle (5,1);
            \fill[blue!40!white] (5,0) rectangle (6,1);
            \fill[blue!40!white] (7,0) rectangle (8,1);
            \fill[blue!40!white] (9,0) rectangle (10,1);
            \fill[blue!40!white] (10,0) rectangle (11,1);
		\node[] at (2.5,0.5) {x};
		\draw[step=1cm,black,very thin] (0, 0) grid (11,1);
	\end{tikzpicture}
	\caption{An example of a jammed configuration on which a predator could build a house. If a house is built on the lot marked with `x', then this house will receive sunlight from the east. It will also block the sunlight to the house to its west, but that does not concern the predator. Notice that the predator could, equally, build a house on the lot to the east of the one marked with `x' (but not on both).}
	\label{fig:Predator_arriving}
\end{figure}
On the other hand, altruists will never (completely) block the sunlight to some other house, but they do not mind if their house does not receive any sunlight (see Figure \ref{fig:Altruist_arriving}).
Clearly, neither type of behavior is desirable to owners of already existing
homes, who built their community according to the rules.
\begin{figure}
	\centering
	\begin{tikzpicture}[scale = 0.65]
		\fill[blue!40!white] (0,0) rectangle (1,1);
            \fill[blue!40!white] (2,0) rectangle (3,1);
            \fill[blue!40!white] (4,0) rectangle (5,1);
            \fill[blue!40!white] (6,0) rectangle (7,1);
            \fill[blue!40!white] (7,0) rectangle (8,1);
            \fill[blue!40!white] (9,0) rectangle (10,1);
            \fill[blue!40!white] (10,0) rectangle (11,1);
		\node[] at (3.5,0.5) {x};
		\draw[step=1cm,black,very thin] (0, 0) grid (11,1);
	\end{tikzpicture}
	\caption{An example of a jammed configuration on which an altruist could build a house. If a house is built on the lot marked with `x', then this house will not block sunlight to its east nor its west neighbor, since the one to the west is exposed to sunlight from the west, and the one to the east is exposed to sunlight from the east. The house built on the lot marked with `x', however, will not receive any sunlight, but that does not concern the altruist.}
	\label{fig:Altruist_arriving}
\end{figure}
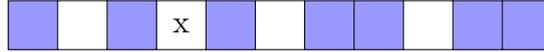

In normal circumstances, both predators and altruists are held in check by
rule-enforcing (or even law-enforcing) authorities. However, there are times
when (and places where) the authorities are either unable or unwilling to
enforce the rules. The causes may vary, from high levels of corruption to
the total societal breakdown. Under such circumstances, it would be
advantageous to live in a community (here modeled by a jammed configuration)
which is not attractive to potential selfish invaders of either type, while
still orderly enough to be acceptable to its rules-respecting and
cooperative inhabitants. Such situations appear in evolutionary biology
when an established community is faced with an invasion of new species, and
is modeled by methods of the evolutionary game theory. So, we borrow some
terms and concepts from that setting.

If all the empty lots on a jammed configuration have the property that, if one builds a house there, this house will be blocked from the sunlight, we call such a configuration \emph{maximal configuration resistant to predators} (see Figure \ref{fig:MCRP}). Similarly, if all the empty lots on a jammed configuration have the property that, if one builds a house there, some other house will not receive any sunlight, we refer to such a configuration as \emph{maximal configuration resistant to altruists} (see Figure \ref{fig:MCRA}). A jammed configuration that is resistant to both predators and altruists is called \emph{evolutionary stable configuration} (see Figure \ref{fig:ES}).
\begin{figure}
\begin{subfigure}{1.0\textwidth}
    \centering
    \begin{tikzpicture}[scale = 0.65]
	\fill[blue!40!white] (0,0) rectangle (1,1);
        \fill[blue!40!white] (2,0) rectangle (3,1);
        \fill[blue!40!white] (4,0) rectangle (5,1);
        \fill[blue!40!white] (6,0) rectangle (7,1);
        \fill[blue!40!white] (7,0) rectangle (8,1);
        \fill[blue!40!white] (9,0) rectangle (10,1);
        \fill[blue!40!white] (10,0) rectangle (11,1);
        \draw[step=1cm,black,very thin] (0, 0) grid (11,1);
    \end{tikzpicture}
    \caption{A maximal configuration resistant to predators.}
    \label{fig:MCRP}
\end{subfigure}

\bigskip

\begin{subfigure}{1.0\textwidth}
    \centering
    \begin{tikzpicture}[scale = 0.65]
	\fill[blue!40!white] (0,0) rectangle (1,1);
        \fill[blue!40!white] (1,0) rectangle (2,1);
        \fill[blue!40!white] (4,0) rectangle (5,1);
        \fill[blue!40!white] (5,0) rectangle (6,1);
        \fill[blue!40!white] (7,0) rectangle (8,1);
        \fill[blue!40!white] (9,0) rectangle (10,1);
        \fill[blue!40!white] (10,0) rectangle (11,1);
	\draw[step=1cm,black,very thin] (0, 0) grid (11,1);
\end{tikzpicture}
    \caption{A maximal configuration resistant to altruists.}
    \label{fig:MCRA}
\end{subfigure}

\bigskip

\begin{subfigure}{1.0\textwidth}
    \centering
    \begin{tikzpicture}[scale = 0.65]
	\fill[blue!40!white] (0,0) rectangle (1,1);
        \fill[blue!40!white] (2,0) rectangle (3,1);
        \fill[blue!40!white] (3,0) rectangle (4,1);
        \fill[blue!40!white] (5,0) rectangle (6,1);
        \fill[blue!40!white] (7,0) rectangle (8,1);
        \fill[blue!40!white] (8,0) rectangle (9,1);
        \fill[blue!40!white] (10,0) rectangle (11,1);
        \draw[step=1cm,black,very thin] (0, 0) grid (11,1);
    \end{tikzpicture}
    \caption{An evolutionary stable configuration.}
    \label{fig:ES}
\end{subfigure}
\caption{Examples of maximal configurations resistant to predators, to altruists, and to both predators and altruists.}
\label{fig:examples_of_MCRP_MCRA_ES}
\end{figure}
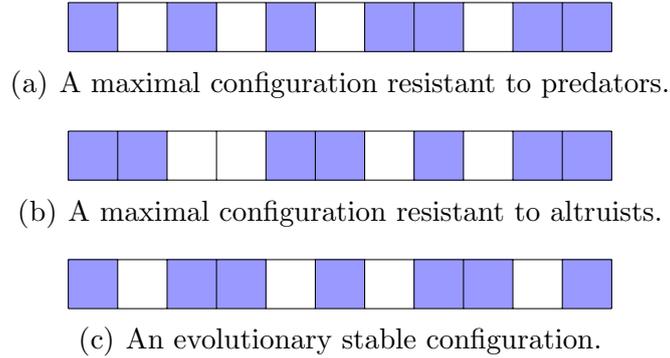

In the rest of this paper, we consider the static model, in which we start
from a fully built jammed configuration and we consider all such configurations
equally likely. There are two main reasons for considering the static and not
the dynamic version. The first is that even the anti-social actors are,
usually, quite rational. There is no reason to violate the rules if you can
still build a house within their scope. So, until a configuration becomes
jammed, everybody behaves cooperatively and follows the rules. The second
reason is that we assume that low-abiding actors are not paranoid. If they
anticipate a breakdown of the social order, their behavior will be influenced
by their fears and they will act following a different set of rules.
Our aim here is not to give a prescription for building resistant communities
(although our results can be used to this end), but to characterize the
resistant communities arising under a given set of rules and to find out
how common they are.

The rest of the paper is organized as follows. In Section \ref{sec:MCRP} we analyze maximal configurations resistant to predators, in Section \ref{sec:MCRA} we are dealing with maximal configurations resistant to altruists, and in Section \ref{sec:ES} we discuss evolutionary stable configurations. The crucial step in the analysis of each of these three cases, is to identify the extra conditions that will secure that a maximal configuration from the Riviera model will additionally be resistant to predators, to altruists, or to both. As mentioned earlier, we develop bivariate generating functions and complexity functions in all three cases. In Section \ref{sec:general} we make some observations about maximal configurations resistant to predators, maximal configurations resistant to altruists, and evolutionary stable configurations in the general setting of the two-dimensional settlement model. It turns out that some very precise observations can be made about the evolutionary stable configurations in two dimensions. Finally, in Section \ref{sec:concluding} we recapitulate our findings and indicate some possible directions of future research.

\section{Maximal configurations resistant to predators}\label{sec:MCRP}
In this, and the following two sections, we represent configurations as binary $0/1$ sequences which are interpreted as sequences of empty/occupied lots on a one-dimensional tract of land. Denote by
\begin{equation}\label{eq:def_of_J_knP}
    J_{k,n}^{P} = \begin{array}{cc}
        \#  \mbox{ of maximal configurations resistant to predators}  \\
         \mbox{of length $n$, with precisely $k$ occupied lots.}
    \end{array}
\end{equation}
Our idea is to use the transfer matrix method (see \cite[\S 4.7]{Stanley}, \cite[\S V]{FlajoletSedgewick}, and \cite[\S 2--4]{SymbDynCoding}) to compute the bivariate generating function for the double sequence $(J_{k,n}^{P})$. This is a well known method for counting words of a regular language. Applicability of this method to our setting relies on the fact that we can check whether a given configuration is in fact a maximal configuration resistant to predators by inspecting only finite size patches of that configuration. There is some freedom while working with the transfer matrix method, and we will use slightly different approaches in this and the following two sections. In this section, we will use the same approach as in the original paper \cite{DPSZ} where the Riviera model was introduced. The first step is to identify the forbidden patterns. To guarantee that we will end up with a jammed configuration, we need to include all the forbidden patterns that were already present in the original Riviera model. These are $111$, $000$, $0100$, and $0010$ (see \cite[Lemma 2.1]{DPSZ} for details). To secure that the maximal configurations that we end up with are also resistant to predators, it is clear that we additionally need to forbid two consecutive zeros. A predator will never build a house on an empty lot that has houses on both of its neighboring lots, so the only option is that we have two consecutive empty lots (we will never have more than two due to maximality). When there are two consecutive empty lots, once the predator builds a house on one of them, her house will still receive sunlight from the side of the other empty lot. As before, we have to pay special attention to boundary lots. Since we assume that the boundary tracts of land are not adjacent to any other buildings, i.e.\ there are no obstructions to sunlight from the boundary, a maximal configuration resistant to predators must have houses on both of its boundary lots. We now have all the necessary information we need to be able to apply the transfer matrix method. Due to the fact that the longest forbidden pattern is of length $4$, we can encode each maximal configuration resistant to predators as a walk on the directed graph shown in Figure \ref{fig:transfer_graph_MCRP}. The vertices of this graph represent all the allowed substrings of length $3$, and the directed edges represent the allowed transitions, see \cite[\S 2.3]{SymbDynCoding} for more details on this construction. There is an edge from the word $u_1u_2u_3$ to $v_1v_2v_3$ if they \emph{overlap progressively}, meaning that $u_2u_3=v_1v_2$, and if the word $u_1u_2u_3v_3=u_1v_1v_2v_3$ is not forbidden.  (Our graph is therefore a subgraph of the $3$-dimensional de Bruijn graph over symbols $\{0,1\}$. Not all edges are present, since the transitions that correspond to the forbidden $4$ letter words must be deleted.) Thus, a transition simply represents the addition of a new lot to the right of the configuration, the state of which is given as the last letter of the string of the target node. The graph in question will be even simpler than the one in the original Riviera model since there we could have two consecutive zeros. It is easy to see (check also \cite[\S 2.1]{DPSZ}) that the only vertices in this graph will be $011$, $110$, $101$ and $010$.
\begin{figure}[h]
    \begin{tikzpicture}[node distance=5em, nodeStyle/.style={draw, circle, minimum size=2.5em}]
	\node (A) [nodeStyle, draw=blue, line width=2pt] {011};
	\node (B) [right of = A, nodeStyle, fill=blue!20] {110};
	\node (C) [right of = B, nodeStyle, fill=blue!20, draw=blue, line width=2pt] {101};
	\node (D) [right of = C, nodeStyle] {010};
	\draw[-Stealth,above] (A) edge[bend left] (B);
	\draw[-Stealth,above] (B) edge[bend left] (C);
	\draw[-Stealth,above] (C) edge[bend left] (A);
        \draw[-Stealth,above] (C) edge[bend left] (D);
	\draw[-Stealth,above] (D) edge[bend left] (C);
    \end{tikzpicture}
    \caption{The transfer digraph for maximal configurations resistant to predators. The starting nodes are \textcolor{blue!70}{shaded} and \textcolor{blue}{thicker outlines} indicate the ending nodes. For example, a maximal configuration resistant to predators $10101011011$ (see Figure \ref{fig:MCRP}) is represented by a walk: 	\\	 $ 101 \to 010 \to 101 \to 010 \to 101 \to 011 \to 110 \to 101 \to 011$.}
    \label{fig:transfer_graph_MCRP}
\end{figure}
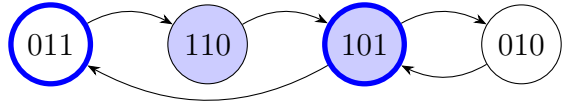
We now define the following matrix function:
\begin{equation}
    \kbordermatrix{
		& \text{011} & \text{110} & \text{101} & \text{010}\\
		\text{011} & 0 & 1 & 0 & 0\\
		\text{110} & 0 & 0 & x & 0\\
		\text{101} & x & 0 & 0 & 1\\
		\text{010} & 0 & 0 & x & 0\\
	} =: A(x).
\end{equation}
The purpose of this matrix function is to encode when a transition results in the increase of number of occupied lots. Namely:
\begin{equation*}
    i \to j \mbox{ is a transition which adds an occupied lot} \iff [A(x)]_{i,j} = x,
\end{equation*}
while the rest of the transitions which do not contribute an occupied lot are set to $1 = x^0$. The powers of $A(x)$, namely $(A(x))^n$, encode the distribution of occupancies for the configurations of length $n$. We have:
\begin{equation}
    [(A(x))^n]_{i,j} = p_0^{i,j} + p_1^{i,j} x + p_2^{i,j} x^2 + \dots + p_n^{i,j} x^n,
\end{equation}
where
\begin{equation*}
    p_k^{i,j} = \begin{array}{cc}
        \#  \mbox{ of walks of length $n$ on the graph in Figure \ref{fig:transfer_graph_MCRP} }  \\
         \mbox{starting with node $i$ and ending with node $j$,} \\
         \mbox{where the number of occupied lots was increased by $1$, $k$ times. }
    \end{array}
\end{equation*}
In order to take into account which vertices we can start with, the number of occupied lots within those vertices, and which vertices we can end with, we define vectors:
\begin{equation}
    a(x) = (0,x^2,x^2,0)^T,\quad  b = (1,0,1,0)^T.
\end{equation}
Combining all of this (and denoting the bivariate generating function for maximal configurations resistant to predators by $F^P(x,y)$), we obtain:
\begin{equation}\label{eq:BGF_for_MCRP}
    \begin{aligned}
	F^P(x, y)
	& = 1 + xy + x^2y^2 + \sum_{n = 3}^{\infty} a(x)^T \cdot (A(x))^{n-3}\cdot b \cdot y^n \\
	& = \frac{1 + xy - (x - x^2)y^2 - x^2y^3}{1 - xy^2 - x^2y^3} \\
	&  = \sum_{n=0}^\infty\sum_{k=0}^\infty J_{k,n}^P x^ky^n,
    \end{aligned}
\end{equation}
where $J_{k,n}^P$ is defined in \eqref{eq:def_of_J_knP}.
\begin{remark}
    Plugging $x = 1$ into \eqref{eq:BGF_for_MCRP}, we get the generating function for the sequence $(J^{P\text{-length}}_n)_n$ which counts the total number of maximal configurations resistant to predators of length $n$. This sequence appears on OEIS \cite{oeis} as \href{https://oeis.org/A000931}{A000931}, the famous Padovan sequence. On the other hand, plugging $y = 1$ into \eqref{eq:BGF_for_MCRP} yields the generating function for the sequence $(J^{P\text{-occupancy}}_k)_k$ which counts the total number of maximal configurations resistant to predators with precisely $k$ houses. This sequence corresponds to the even more famous Fibonacci sequence (which can be found on OEIS under code \href{https://oeis.org/A000931}{A000045}).
\end{remark}
Our next goal is to obtain the complexity function of maximal configurations resistant to predators. We first recall the definition of complexity function (configurational entropy) of a certain model.
\begin{definition}\label{def:complexity}
    For a fixed density $\rho \ge 0$, take $((k_i, n_i))_i$ to be any sequence of pairs of non-negative integers such that $\lim_{i\to\infty} n_i = +\infty$ and  $\lim_{i\to\infty} \frac{k_i}{n_i} = \rho$. Denote by $J_{k_i,n_i}$ the number of configurations of length $n_i$ with density $\frac{k_i}{n_i}$. We are interested in the quantity
	$$\limsup_{i\to \infty} \frac{\ln J_{k_i,n_i}}{n_i},$$
    which is the exponential rate of growth of these configurations. If we now take the supremum over all such sequences, we arrive at the definition of complexity function $\FJADEF{S(\rho)}{[0,\infty)}{[0,\infty)}$
    \begin{equation}\label{eq:cmplxDEF}
	S(\rho) = \sup_{(k_i,n_i)} \limsup_{i\to \infty} \frac{\ln J_{k_i,n_i}}{n_i},
    \end{equation}
    where the supremum runs over all the sequences such that $k_i/n_i\to\rho$ and $n_i\to\infty$.
\end{definition}
\begin{remark}
    Notice that we use the notation $J_{k, n}$ for the number of configurations of length $n$ and occupancy $k$, in an arbitrary model of this sort, regardless of the background rule for composing such configurations. We will later compute the complexity functions for maximal configurations resistant to predators, maximal configurations resistant to altruists, and maximal evolutionary stable configurations; but one can try to calculate the above defined complexity function for any one-dimensional irreversible deposition model, or, indeed, for any ensemble of binary strings.
\end{remark}
\begin{remark}
    Whenever we encounter $J_{k,n}=0$ for some $(k,n)$, we will redefine it as $J_{k,n}=1$ so that $\ln J_{k,n}=0$ can be computed. Consequentially, if there are no configurations with densities approaching a certain $\rho$, we get $S(\rho)=0$. Also note that the $\limsup$ can be replaced with $\lim$ since we can, if needed, pass to a subsequence.
\end{remark}
\begin{remark}
    Definition \ref{def:complexity} implies that the number of configurations with density $k/n\approx\rho$ grows as $e^{nS(\rho)}$ for large $n$. The density $\rho_{\star}$ at which the complexity function $S(\rho)$ attains its maximum, i.e.\ the density corresponding to the largest rate of growth, is called the \emph{equilibrium density}.
\end{remark}
\begin{remark}
    In most commonly encountered models the $\sup$ in the definition is superfluous, as any choice of the sequence (say $((k_n,n))_n$ where $k_n=\FLOOR{\rho n}$) will produce the same limit.
\end{remark}
We are now ready to state the main result of this section.
\begin{theorem}\label{tm:cmpl_predatori}
    The complexity function $S^P(\rho)$ of maximal configurations resistant to predators is
    \begin{equation}\label{eq:cmpl_predatori}
        S^P(\rho) = (1 - \rho) \ln(1 - \rho) - (2\rho - 1) \ln(2\rho - 1) - (2 - 3\rho) \ln(2 - 3\rho), \quad \frac{1}{2} < \rho < \frac{2}{3}.
    \end{equation}
\end{theorem}
\begin{remark}
    Since the maximal configurations resistant to predators form a subset of all the maximal configurations from the Riviera model, we clearly have that the support of the complexity function of maximal configurations resistant to predators is a subset of the support of the complexity function of the Riviera model. However, it is easy to see that the two supports will in fact coincide. Notice that the configurations of the form
    \begin{equation*}
        1010101 \ldots 010101 \qquad \textnormal{or} \qquad 110110 \ldots 11011
    \end{equation*}
    are maximal and resistant to predators. In the limit, the pattern on the left gives the density of $\frac{1}{2}$, and the pattern on the right gives the density of $\frac{2}{3}$. Clearly, by combining these two patterns, we can achieve any density between $\frac{1}{2}$ and $\frac{2}{3}$.
\end{remark}
The configurational entropy of jammed configurations is usually determined either by means of direct combinatorial reasoning \cite{17_crisanti2000inherent, 31_dean2000metastable, 32_lefevre2001metastable, DPSZ-Ryd}, or by using the transfer-matrix approach \cite{20_lefevre2001tapping, 22_de2002jamming}. Recently, a new method for determining complexity function has been developed in \cite{KL}, inspired by the theory of renewal processes. Since our transfer matrix encodes transitions which always add one more lot to the right end of the configuration that is being built, we could use the transfer-matrix approach to compute complexity. This approach would include computing the characteristic equation of the transfer matrix $A(x)$, and finding the appropriate rational parametrization of this equation. For details, see \cite{krapivsky2022jamming} where this approach was used to obtain the complexity function of the original Riviera model. We will prove Theorem \ref{tm:cmpl_predatori} in three different ways in order to stress the combinatorial simplicity of maximal configurations resistant to predators, to make some informative connections between maximal configurations resistant to predators and one of the most famous models for irreversible deposition in the literature, and to illustrate the newly developed method from \cite{KL} (since we will again use it in the following sections).
\begin{proof}[Proof of Theorem \ref{tm:cmpl_predatori} using the direct combinatorial approach]
    The idea is to find a closed formula for the value $J_{k,n}^P$, i.e.\ for the number of maximal configurations resistant to predators of length $n$ with precisely $k$ occupied lots. There is an easy procedure to construct a maximal configuration resistant to predators of length $n$, with $k$ occupied lots, and by explaining this procedure it will become evident how many such configurations there are. Clearly, for some values of $k$ and $n$ we have $J_{k,n}^P = 0$, but the formula will work even in those cases, as long as we take $k \le n$ (which we always have since the number of houses is obviously bounded by the number of lots). A maximal configuration resistant to predators that has length $n$, and $k$ of those $n$ lots are occupied by a house, has precisely $n-k$ empty lots. As explained earlier, all the empty lots in a maximal configuration resistant to predators have to be isolated (if there are two consecutive empty lots, a predator could build a house on one of them). Moreover, recall that we have to start and end with occupied lots (if the boundary lots are not occupied, predators could come there since there is no obstruction to sunlight from the boundary).

    The idea is to start with a configuration that has precisely $n-k$ empty lots, and in which occupied and empty lots alternate (we start and end with an occupied lot). This configuration can be represented as
    \begin{equation}\label{eq:starting_conf_MCRP}
        \underbrace{1010101 \ldots 010101}_{n-k \textnormal{ empty lots, and } n - k + 1 \textnormal{ occupied lots}}
    \end{equation}
    This configuration is already a maximal configuration resistant to predators, but we may need to insert more occupied lots to reach the total number of $k$ occupied lots (and consequentially, to reach the length $n$). Notice that we can insert one additional occupied lot to the left of any of the existing occupied lots. In that way we keep the number of empty lots fixed ($n-k$), we keep all the empty lots isolated, and we do not tamper with permissibility since all the blocks of two consecutive houses will be exposed to sunlight from east and west. Clearly, we cannot lose maximality or resistance to predators by inserting more houses in this way. Hence, we have $n-k+1$ potential places for inserting new houses. Since the configuration shown in \eqref{eq:starting_conf_MCRP} already has $n-k+1$ houses, we still need to insert $k - (n - k + 1) = 2k-n-1$ new houses to have $k$ of them. This means that out of $n-k+1$ possible options, we have to choose exactly $2k-n-1$ of them. Therefore,
    \begin{equation}\label{eq:exact_form_for_J_knP}
        J_{k,n}^P = \binom{n-k+1}{2k-n-1}.
    \end{equation}
    Let us now fix some $\rho \ge 0$, and take an arbitrary sequence $((k_i,n_i))_i$ of pairs of non-negative integers such that $\lim_{i\to \infty} n_i = +\infty$, and $\lim_{i \to \infty} \frac{k_i}{n_i} = \rho$. Recall that the Stirling's approximation gives us
    \begin{equation*}
        \ln(n!) = n \ln n - n + O(\ln n),
    \end{equation*}
    where the big $O$ notation means that, for all sufficiently large values of $n$, the difference between $\ln(n!)$ and $n \ln n - n$ will be at most proportional to $\ln n$. Combining Stirling's approximation with \eqref{eq:exact_form_for_J_knP} gives us
    \begin{align*}
        \lim_{i \to \infty} \frac{\ln (J_{k_i, n_i}^P)}{n_i}
        & = \lim_{i \to \infty} \frac{{\ln((n_i-k_i+1)!) - \ln((2k_i-n_i-1)!) - \ln((2n_i-3k_i+2)!)}}{n_i} \\
        & = \lim_{i \to \infty} \left\{ \frac{n_i-k_i+1}{n_i} \cdot \left[ \ln \left(n_i \cdot \frac{n_i-k_i+1}{n_i} \right) - 1 \right] \right. \\
        & \qquad \qquad - \frac{2k_i-n_i-1}{n_i} \cdot \left[ \ln \left(n_i \cdot \frac{2k_i-n_i-1}{n_i} \right) - 1 \right] \\
        & \qquad \qquad \qquad - \left. \frac{2n_i-3k_i+2}{n_i} \cdot \left[ \ln \left(n_i \cdot \frac{2n_i-3k_i+2}{n_i} \right) - 1 \right] \right\}\\
        & = \lim_{i \to \infty} \left\{ \frac{n_i-k_i+1}{n_i} \cdot \ln \left( \frac{n_i-k_i+1}{n_i} \right) \right.\\
        & \qquad \qquad - \frac{2k_i-n_i-1}{n_i} \cdot \ln \left( \frac{2k_i-n_i-1}{n_i} \right) \\
        & \qquad \qquad \qquad - \left. \frac{2n_i-3k_i+2}{n_i} \cdot \ln \left( \frac{2n_i-3k_i+2}{n_i} \right) \right\} \\
        & = (1 - \rho) \ln(1 - \rho) - (2\rho - 1) \ln(2\rho - 1) - (2 - 3\rho) \ln(2 - 3\rho).
    \end{align*}
    Combining this with \eqref{eq:cmplxDEF} finishes the proof.
\end{proof}
\begin{proof}[Proof of Theorem \ref{tm:cmpl_predatori} using the connection with Flory's model]
    Probably the most famous model of irreversible deposition of particles is Flory's model (see \cite{Flory, Krapivsky_et_al, KL, DPSZ-Ryd, Page, GerinPRparkingHAL}), where atoms are deposited on a one-dimensional lattice. Each site in the lattice can be occupied by an atom, or left vacant. The only constraint is that each atom has to have vacant sites to both of its sides (i.e.\ no two consecutive sites can both be occupied by atoms). Notice that this is exactly the opposite of the condition characterizing maximal configurations resistant to predators where it is not allowed to have two consecutive empty lots. This implies that maximal configurations resistant to predators are ``negatives'' of the jammed configurations in Flory's model. The correspondence is not one-to-one, since the jammed configurations in Flory's model can start (or end) with an occupied site, and maximal configurations resistant to predators cannot start (nor end) with an empty lot. However, this only means that the number of maximal configurations resistant to predators and the number of jammed configurations in Flory's model differ by at most a constant factor, which does not affect the complexity functions (since the complexity function ignores sub-exponential factors). The complexity function for Flory's model is already known in the literature (see for example \cite[formula (7.20)]{Krapivsky_et_al} or \cite[Remark 3.7]{DPSZ-Ryd}), and the formula is
    \begin{equation*}
        S^{\text{Flory}}(\rho) = \rho \ln \rho - (1-2\rho) \ln (1-2\rho) - (3\rho - 1)\ln(3\rho - 1), \quad \frac{1}{3} < \rho < \frac{1}{2}.
    \end{equation*}
    Since atoms in Flory's model correspond to empty lots in our model, and vacant sites from Flory's model correspond to occupied lots in our model, we have
    \begin{equation*}
        S^P(\rho) = S^{\text{Flory}}(1 - \rho) = (1 - \rho) \ln(1 - \rho) - (2\rho - 1) \ln(2\rho - 1) - (2 - 3\rho) \ln(2 - 3\rho),
    \end{equation*}
    where $\frac{1}{3} < 1 - \rho < \frac{1}{2}$, i.e.\ $\frac{1}{2} < \rho < \frac{2}{3}$.
\end{proof}
\begin{proof}[Proof of Theorem \ref{tm:cmpl_predatori} using the method developed in \cite{KL}]
    To compute the complexity function of maximal configurations resistant to predators by using the method developed in \cite{KL}, we first need to have the bivariate generating function for the sequence $(J_{k,n}^P)$. Luckily, we already computed this in \eqref{eq:BGF_for_MCRP}. Denote the denominator in \eqref{eq:BGF_for_MCRP} by $q^P(x, y) = 1 - xy^2 - x^2y^3$. The formula for the complexity function is given by
    \begin{equation}\label{eq:formula_for_CMPL_KL}
        S^P(\rho) = -\rho \ln x_0 - \ln y_0,
    \end{equation}
    where the connection between $\rho$, $x_0$ and $y_0$ is given by
    \begin{equation*}
        q^P(x_0, y_0) = 0 \quad \textnormal{and} \quad \rho = \left[ \frac{x}{y} \frac{\partial_x q^P}{\partial_y q^P} \right]_{x = x_0, y = y_0}.
    \end{equation*}
    Notice first that from $q^P(x_0, y_0) = 0$ we have
    \begin{equation}\label{eq:cpml_MCRP_step1}
        1 = x_0y_0^2 + x_0^2y_0^3 = y_0 \cdot x_0y_0 \cdot(1 + x_0y_0).
    \end{equation}
    Next, we have
    \begin{equation}\label{eq:cpml_MCRP_step2}
        \rho = \frac{x_0}{y_0} \cdot \frac{(-y_0^2 - 2x_0y_0^3)}{(-2x_0y_0 - 3x_0^2y_0^2)} = \frac{1 + 2x_0y_0}{2+3x_0y_0}.
    \end{equation}
    From \eqref{eq:cpml_MCRP_step2} we easily get
    \begin{equation}\label{eq:cpml_MCRP_step3}
        x_0y_0 = \frac{2\rho - 1}{2 - 3\rho}.
    \end{equation}
    By combining \eqref{eq:cpml_MCRP_step1} and \eqref{eq:cpml_MCRP_step3} we get
    \begin{equation*}
        1 = y_0 \cdot \frac{2\rho - 1}{2 - 3\rho} \cdot \left( 1 + \frac{2\rho - 1}{2 - 3\rho} \right) = y_0 \cdot \frac{2\rho - 1}{2 - 3\rho} \cdot \frac{1-\rho}{2-3\rho} = y_0 \cdot \frac{(1-\rho)(2\rho-1)}{(2-3\rho)^2},
    \end{equation*}
    which implies
    \begin{equation}\label{eq:cpml_MCRP_step4}
        y_0 = \frac{(2-3\rho)^2}{(1-\rho)(2\rho-1)}.
    \end{equation}
    Relation \eqref{eq:cpml_MCRP_step4}, together with \eqref{eq:cpml_MCRP_step3} gives
    \begin{equation}\label{eq:cpml_MCRP_step5}
        x_0 = \frac{(1-\rho)(2\rho - 1)^2}{(2 - 3\rho)^3}.
    \end{equation}
    Finally, combining \eqref{eq:formula_for_CMPL_KL}, \eqref{eq:cpml_MCRP_step4} and \eqref{eq:cpml_MCRP_step5} implies
    \begin{align*}
        S^P(\rho)
        & = -\rho (\ln(1-\rho) + 2\ln(2\rho-1) - 3\ln(2-3\rho))\\
        & \qquad - 2\ln(2-3\rho) + \ln(1-\rho) + \ln(2\rho-1) \\
        & = (1-\rho)\ln(1 - \rho) - (2\rho - 1) \ln(2\rho - 1) - (2 - 3\rho) \ln(2 - 3\rho).
    \end{align*}
\end{proof}
\begin{remark}
    For a slightly more formal explanation of the method for computing the complexity function introduced in \cite{KL}, see \cite{PSZ-CroCo}.
\end{remark}

\begin{figure}
	\includegraphics[scale=.7]{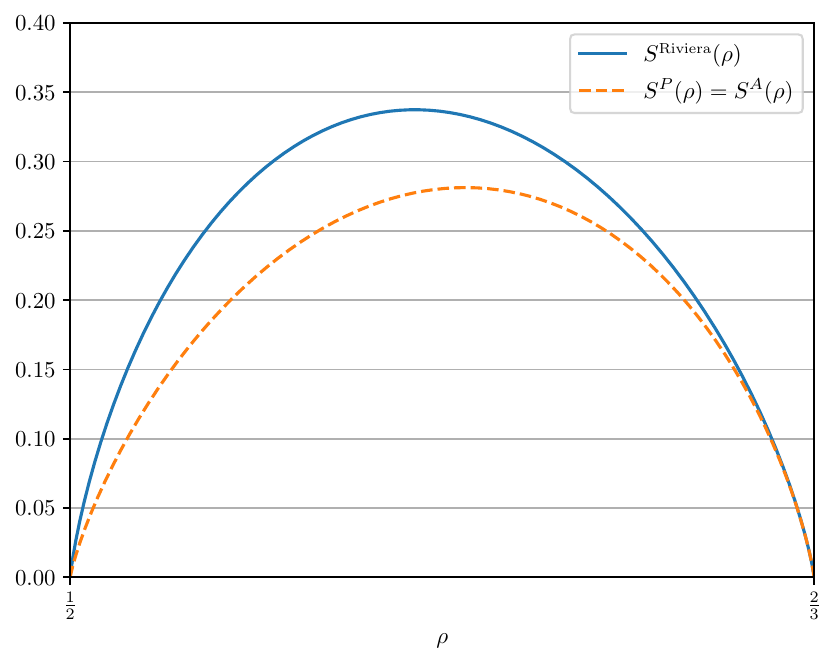}
	\caption{The complexity function $S^P(\rho)$ of maximal configurations resistant to predators compared to the complexity function of the Riviera model.}\label{fig:Complexity_P/A_vs_Riviera}
\end{figure}

\begin{remark}
	The complexity function $S^P(\rho)$ calculated in Theorem \ref{tm:cmpl_predatori} is shown in Figure \ref{fig:Complexity_P/A_vs_Riviera} where it is compared to the complexity function of the full Riviera model, as computed in \cite{krapivsky2022jamming}.
\end{remark}

\section{Maximal configurations resistant to altruists}\label{sec:MCRA}
The main goal of this section is to compute the bivariate generating function, and then the complexity function, of maximal configurations resistant to altruists. For computing the bivariate generating function, we again use the transfer matrix method. Before we proceed, we briefly come back to the original Riviera model. In \cite{DPSZ}, the bivariate generating function for the Riviera model was computed in an analogous way as the bivariate generating function for maximal configurations resistant to predators in the previous section. However, instead of adding only one new lot on the right hand side of the configuration with every step, we can add a whole block of lots. Let us explain this on the example of the original Riviera model, and then with a slight modification, we will be able to directly apply it to maximal configurations resistant to altruists. We consider the blocks that start with empty lots, and finish with occupied lots. Due to the permissibility condition, we cannot have more than two consecutive houses, and due to the maximality condition, we cannot have more than two consecutive empty lots. This leaves us with four possible blocks: $01$, $001$, $011$, $0011$. Out of these four blocks, only the block $001$ is problematic, because once we glue another one of these blocks to it, we will always end up with a forbidden pattern $0010$ (see \cite[Lemma 2.1]{DPSZ}). The other blocks do not necessarily lead to the same problem, and (as always with the transfer matrix method) we can solve other potential problems by forbidding some particular transitions. Hence, jammed configurations in the original Riviera model are built from the blocks
\begin{equation}\label{eq:building_blocks_for_Riviera}
    01, \quad 011, \quad \textnormal{and} \quad 0011.
\end{equation}
As usual, we need to take extra care when it comes to the beginning, and the end of jammed configurations. To get the starting blocks, we just need to remove the first (empty) lot from each of the blocks shown in \eqref{eq:building_blocks_for_Riviera}. We can end with any of the blocks from \eqref{eq:building_blocks_for_Riviera}, but in addition to the block $011$, we can also end with $0110$, and similarly, beside the block $0011$, we have to take into account the block $00110$ as a possible ending block. Using the order of blocks from \eqref{eq:building_blocks_for_Riviera}, we have the following formula for the bivariate generating function for the original Riviera model
\begin{equation}\label{eq:BGF_Riviera_aux}
    F(x,y) = 1 + \sum_{k = 0}^{\infty}
    \begin{bmatrix}
        xy & x^2y^2 & x^2y^3
    \end{bmatrix}
    \cdot
    \begin{bmatrix}
        xy^2 & x^2y^3 & 0 \\
        xy^2 & x^2y^3 & x^2y^4 \\
        xy^2 & x^2y^3 & x^2y^4
    \end{bmatrix}^k
    \cdot
    \begin{bmatrix}
        1 \\
        1 + y \\
        1 + y
    \end{bmatrix},
\end{equation}
where $x$ is again a formal variable corresponding to the number of houses, and $y$ is a formal variable corresponding to the length of the configuration. Notice that the only forbidden step is from block $01$ to block $0011$, since this transition would form the forbidden pattern $0100$. The expression from \eqref{eq:BGF_Riviera_aux} can easily be evaluated to obtain
\begin{equation*}
    F(x,y) = \frac{1+xy-(x-x^2)y^2+x^2y^3-x^3y^5}{1-xy^2-x^2y^3-x^2y^4+x^3y^6},
\end{equation*}
which reconstructs the result from \cite[p.\ 9]{DPSZ}.

We now move on to configurations that are of main interest in this chapter, and those are the maximal configurations resistant to altruists. To be able to apply the transfer matrix method, we have to detect the forbidden patterns. As in the previous section, it is easy to notice what kind of patterns will allow altruists to build a house. Since altruists do not want to be the reason that some house becomes (completely) blocked from the sunlight, they will never build a house next to a block of two consecutive houses. Notice also that if there is a block of two consecutive empty lots in the jammed configuration, it has to be surrounded with blocks of two consecutive houses from both sides since some of the forbidden patterns (000, 0010, 0100) would appear if this was not the case. This means that the only possibility is that altruists arrive on an empty lot that has houses on both of its neighboring lots (to the east and west). Furthermore, both of those houses need to receive sunlight even after the altruist comes. This implies that the only additional forbidden pattern that we need to take into account is $01010$. Therefore, we would need to work with blocks of length $4$ if we would want to apply the same approach as in the previous chapter. This can be done, but it is much easier to notice that this new forbidden pattern causes minimal changes to the calculation performed in \eqref{eq:BGF_Riviera_aux}, namely, we only need to forbid one more transition - the one from $01$ to $01$. Notice also that the maximal configurations resistant to altruists are not allowed to start with $1010$ nor to end with $0101$, because the sun also comes from the boundary of the tract of land. However, by forbidding the transition from $01$ to $01$, we immediately take care of these boundary restrictions. Hence, we have
\begin{align}\label{al:BGF_MCRA}
    F^A(x,y)
    & = 1 + \sum_{k = 0}^{\infty}
    \begin{bmatrix}
        xy & x^2y^2 & x^2y^3
    \end{bmatrix}
    \cdot
    \begin{bmatrix}
        0 & x^2y^3 & 0 \\
        xy^2 & x^2y^3 & x^2y^4 \\
        xy^2 & x^2y^3 & x^2y^4
    \end{bmatrix}^k
    \cdot
    \begin{bmatrix}
        1 \\
        1 + y \\
        1 + y
    \end{bmatrix} \nonumber\\
    & = \frac{1+xy+x^2y^2+x^2y^3+x^3y^4}{1-x^2y^3-x^2y^4-x^3y^5}
 	  = \frac{1+xy+x^2y^2+x^2y^3+x^3y^4}{(1-xy^2-x^2y^3)(1+xy^2)},
\end{align}
where we denoted the bivariate generating function for maximal configurations resistant to altruists by $F^A(x, y)$.
\begin{remark}
    As in the case of maximal configurations resistant to predators, if we plug $x = 1$, or $y = 1$ into \eqref{al:BGF_MCRA}, we get some sequences that are already present in the OEIS (though not as popular as the ones from the previous section). Plugging $x = 1$ into \eqref{al:BGF_MCRA}, we get the generating function for the sequence $(J_n^{A\text{-length}})_n$ which counts the total number of maximal configurations resistant to altruists of length $n$. This sequence appears on OEIS as \href{https://oeis.org/A017818}{A017818} and counts compositions of an integer $n$ into parts $3$, $4$, and $5$. On the other hand, plugging $y = 1$ into \eqref{al:BGF_MCRA} gives us the generating function for the sequence $(J_k^{A\text{-occupancy}})_k$ which counts the total number of maximal configurations resistant to altruists with precisely $k$ houses. This sequence can be found on OEIS under the code \href{https://oeis.org/A008346}{A008346}, and it is again related to Fibonacci sequence. More precisely, this is the sequence where the term $(-1)^n$ is added to the $n$-th Fibonacci number.
\end{remark}
Now that we have obtained the bivariate generating function, we are ready to calculate the complexity function of maximal configurations resistant to altruists.
\begin{theorem}\label{tm:cmpl_altruisti}
    The complexity function $S^A(\rho)$ of maximal configurations resistant to altruists is given by
    \begin{equation}
        S^A(\rho) = (1 - \rho) \ln(1 - \rho) - (2\rho - 1) \ln(2\rho - 1) - (2 - 3\rho) \ln(2 - 3\rho), \quad \frac{1}{2} < \rho < \frac{2}{3}.
    \end{equation}
\end{theorem}
\begin{remark}
    Notice that the complexity function $S^A(\rho)$ is exactly the same as the complexity function $S^P(\rho)$ (see Theorem \ref{tm:cmpl_predatori} and Figure \ref{fig:Complexity_P/A_vs_Riviera}). Even though there is an intuitive argument why is that so (see Remark \ref{rem:pred=alt}), it is not obvious at first that these two models are equivalent in this sense.
\end{remark}
\begin{remark}
    As in the case of maximal configurations resistant to predators, it is easy to see that the support of the complexity function of maximal configurations resistant to altruists is indeed $\left[ \frac{1}{2}, \frac{2}{3} \right]$, the same as in the case of the full Riviera model. The extremal configurations resistant to altruists are
    \begin{equation*}
        11001100 \ldots 110011 \qquad \textnormal{and} \qquad 110110110 \ldots 11011
    \end{equation*}
    In the limit, the pattern on the left gives the density of $\frac{1}{2}$, and the pattern on the right gives the density of $\frac{2}{3}$. Clearly, by combining these two patterns, we can achieve any density between $\frac{1}{2}$ and $\frac{2}{3}$.
\end{remark}
Theorem \ref{tm:cmpl_altruisti} can be proved in a straightforward way by
using the approach developed in \cite{KL}. Hence, we omit the proof.
\begin{remark}\label{rem:pred=alt}
    Analogously as in the previous section, denote by
    \begin{equation}\label{eq:def_of_J_knA}
        J_{k,n}^{A} = \begin{array}{cc}
        \#  \mbox{ of maximal configurations resistant to altruists}  \\
         \mbox{of length $n$, with precisely $k$ occupied lots.}
        \end{array}
    \end{equation}
    Directly from \eqref{al:BGF_MCRA} we can read the recurrence relation for the double sequence $(J_{k,n}^{A})$, namely
    \begin{equation}\label{eq:rec_for_altruists}
        J_{k, n}^A = J_{k-2,n-3}^A + J_{k-2,n-4}^A + J_{k-3,n-5}^A.
    \end{equation}
    Similarly, from \eqref{eq:BGF_for_MCRP} we have
    \begin{equation*}
        J_{k, n}^P = J_{k-1,n-2}^P + J_{k-2,n-3}^P.
    \end{equation*}
    By applying this recurrence relation once more to the term $J_{k-1,n-2}^P$, we get
    \begin{equation}\label{eq:rec_for_predators}
        J_{k, n}^P = J_{k-2,n-3}^P + J_{k-2,n-4}^P + J_{k-3,n-5}^P.
    \end{equation}
    Notice now that recurrence relations \eqref{eq:rec_for_altruists} and \eqref{eq:rec_for_predators} are the same. Due to the different initial conditions, it does not hold that $J_{k,n}^A = J_{k,n}^P$ for $k, n \in \mathbb{N}$. However, it can be shown that the exponential growth for both double sequences is the same, and different initial conditions only imply that those sequences differ by a multiplicative factor. Since the complexity function ignores sub-exponential factors, the complexity function $S^A(\rho)$ of maximal configurations resistant to altruists is the same as the complexity function $S^P(\rho)$ of maximal configurations resistant to predators. Closely related to the observations above is the fact that the denominator of \eqref{eq:BGF_for_MCRP} divides the denominator of \eqref{al:BGF_MCRA} since the latter factorizes as
    $$1-x^2y^3-x^2y^4-x^3y^5 = (1-xy^2-x^2y^3)(1+xy^2).$$
\end{remark}
\section{Evolutionary stable configurations}\label{sec:ES}
In this section we combine the concepts from the last two sections. We are interested in the maximal configurations resistant to both predators and altruists. We call such configurations evolutionary stable, since even the newly introduced agents cannot invade them. We again start with the computation of the bivariate generating function. Notice that here (beside the forbidden patterns that are already present in the original Riviera model: $111$, $000$, $0100$ and $0010$) the forbidden patterns include both $00$ (which is forbidden because of the predators) and $01010$ (which is forbidden because of the altruists). We will once again use the transfer matrix method, and we can choose between the approaches used in the previous two sections. Both are reasonably simple, but the one where we glue blocks (and not just add one new lot) in each step is, again, much simpler. Since the evolutionary stable configurations are jammed configurations from the Riviera model, we can use the same reasoning as in the last section when we were discussing the maximal configurations resistant to altruists. We know that the jammed configurations in the original Riviera model are composed of blocks $01$, $011$, and $0011$. Clearly, we cannot use the block $0011$ to form evolutionary stable configurations since it contains the forbidden pattern $00$. Hence, the evolutionary stable configurations are composed only from the blocks $01$ and $011$. Furthermore, to ensure their resistance to altruists, we (as in the previous section) need to forbid the transition from the block $01$ to itself. When it comes to starting and ending blocks, the situation is even simpler than before. We know that we are not allowed to start (nor end) with an empty lot (due to the predators). Therefore, to obtain the starting nodes, we have to delete the first (empty) lot from the blocks $01$ and $011$. We do not need to pay special attention to the ending nodes, since jammed configurations can end with any of the two building blocks ($01$ or $011$). This leaves us with the following expression for the bivariate generating function for the evolutionary stable configurations
\begin{align}\label{al:BGF_ES}
    F^{ES}(x,y)
    & = 1 + \sum_{k = 0}^{\infty}
    \begin{bmatrix}
        xy & x^2y^2
    \end{bmatrix}
    \cdot
    \begin{bmatrix}
        0 & x^2y^3 \\
        xy^2 & x^2y^3
    \end{bmatrix}^k
    \cdot
    \begin{bmatrix}
        1 \\
        1
    \end{bmatrix} \nonumber\\
    & = \frac{1+xy+x^2y^2-x^2y^3+x^3y^4-x^3y^5}{1-x^2y^3-x^3y^5},
\end{align}
where we denoted the bivariate generating function for the evolutionary stable configurations by $F^{ES}(x, y)$.
\begin{remark}
    Let us again inspect what particular sequences we obtain if we look separately at the total number of evolutionary stable configurations of particular length, and the total number of evolutionary stable configurations with a prescribed number of houses. Plugging $x = 1$ into \eqref{al:BGF_ES}, we get the generating function for the  sequence $(J_n^{ES\text{-length}})_n$ which counts the total number of evolutionary stable configurations of length $n$. This sequence appears on OEIS as \href{https://oeis.org/A052920}{A052920}, and counts the compositions of an integer $n$ into parts $3$ and $5$. When we discussed (in the previous sections) these kinds of relations between sequences that appear in our model, and sequences already present on the OEIS, we did not go into details about the offset. However, here we comment on it since we want to make an additional observation about the evolutionary stable configurations. Notice first that there is only one evolutionary stable configuration of length $1$ ($1$), and only one of length $2$ ($11$). Then something interesting happens as there are no evolutionary stable configurations of length $3$. There are, however, three jammed configurations of length $3$ in the Riviera model: $011$, $101$, and $110$. The configurations $011$ and $110$ are not resistant to predators, and the configuration $101$ is not resistant to altruists. It is also easy to check that there are two evolutionary stable configurations of length $4$ ($1101$ and $1011$). There is one additional maximal configuration of length $4$ in the Riviera model ($0110$), but this one is not resistant to predators. Hence, our sequence starts with $1, 1, 0, 2$. Inspecting the sequence \href{https://oeis.org/A052920}{A052920}, we see that this means that there is an offset of $5$ present. More precisely, the first element ($1$) corresponds to the number of ways $5$ can be composed into parts $3$ and $5$ ($5 = 5$). The second element (again $1$) corresponds to $6 = 3+3$. The zero corresponds to the fact that $7$ cannot be composed from $3$ and $5$. Then we have that $8 = 3+5 = 5+3$. The question is whether there are bigger numbers than $7$ that cannot be composed from $3$ and $5$. Even though it is easy to see from the recurrence relation that the answer is no, we could apply here the famous Chicken McNugget Theorem which states that for any two relatively prime positive integers $m$ and $n$, the greatest integer that cannot be written in the form $am + bn$, for nonnegative integers $a$ and $b$, is $mn-m-n$. In our case, the greatest integer that cannot be composed into parts $3$ and $5$ is $3\cdot 5 - 3 - 5 = 7$. Translated to our problem, this means that the length $3$ is the only length for which no evolutionary stable configurations exist.

    On the other hand, plugging $y = 1$ into \eqref{al:BGF_ES} gives us the generating function for the sequence $(J_k^{ES\text{-occupancy}})_k$ which counts the total number of evolutionary stable configurations with precisely $k$ houses. This is again the famous Padovan sequence (\href{https://oeis.org/A000931}{A000931}).
\end{remark}
Now we move to the main result of this section, and this is again the explicit form for the complexity function of the evolutionary stable configurations.
\begin{theorem}\label{tm:cmpl_ES}
    The complexity function $S^{ES}(\rho)$ of evolutionary stable configurations is given by
    \begin{equation}
        S^{ES}(\rho) = (2\rho - 1) \ln(2\rho - 1) - (2 - 3\rho) \ln(2 - 3\rho) - (5\rho - 3) \ln(5\rho - 3), \quad \frac{3}{5} < \rho < \frac{2}{3}.
    \end{equation}
\end{theorem}
\begin{remark}
    Unlike in the previous two sections, the support for the complexity function of evolutionary stable configurations does not match the support of the complexity function of the original Riviera model. The upper bound is the same, and can again be achieved by the same pattern as in the previous two sections, namely $110110110 \ldots 11011$. But notice that the pattern which achieved the density $\frac{1}{2}$ and was resistant to predators is not resistant to altruists and vice versa. However, it is again easy to see why the smallest possible density is $\frac{3}{5}$. As we already explained, the only building blocks appearing in evolutionary stable configurations are $01$ and $011$, and we have to forbid the transition from the block $01$ to itself. This implies that the sparsest evolutionary stable configuration will be the one where these two blocks alternate, i.e.\ the one with the pattern $1011010110 \ldots 1011$. It is easy to see that, in the limit, this pattern gives the density of $\frac{3}{5}$.
\end{remark}
As the complexity function of Theorem \ref{tm:cmpl_ES} can be routinely
computed by using the method developed in \cite{KL}, we omit the proof.

\begin{figure}
	\begin{subfigure}{.48\textwidth}
		\includegraphics[scale=.5]{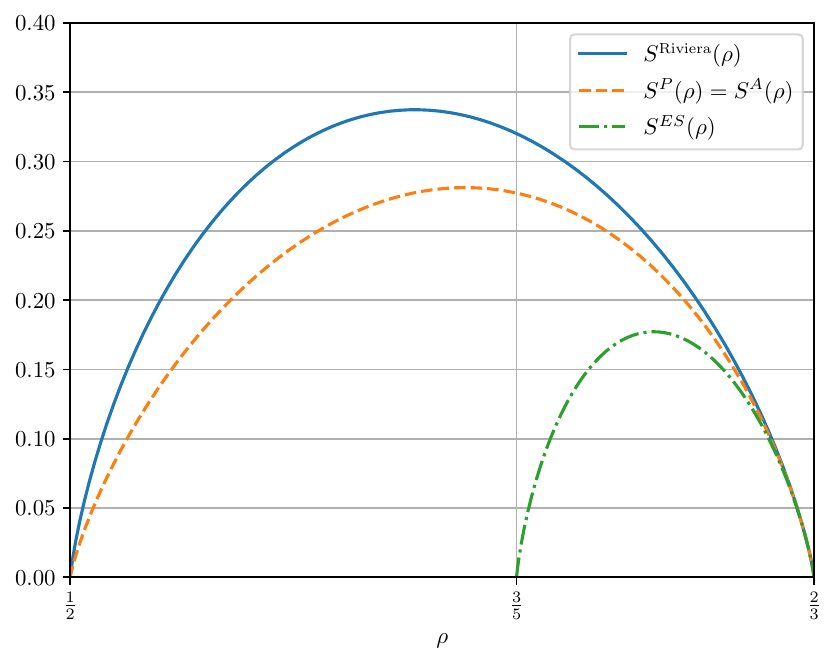}
		\caption{}
	\end{subfigure}
	\begin{subfigure}{.48\textwidth}
		\includegraphics[scale=.5]{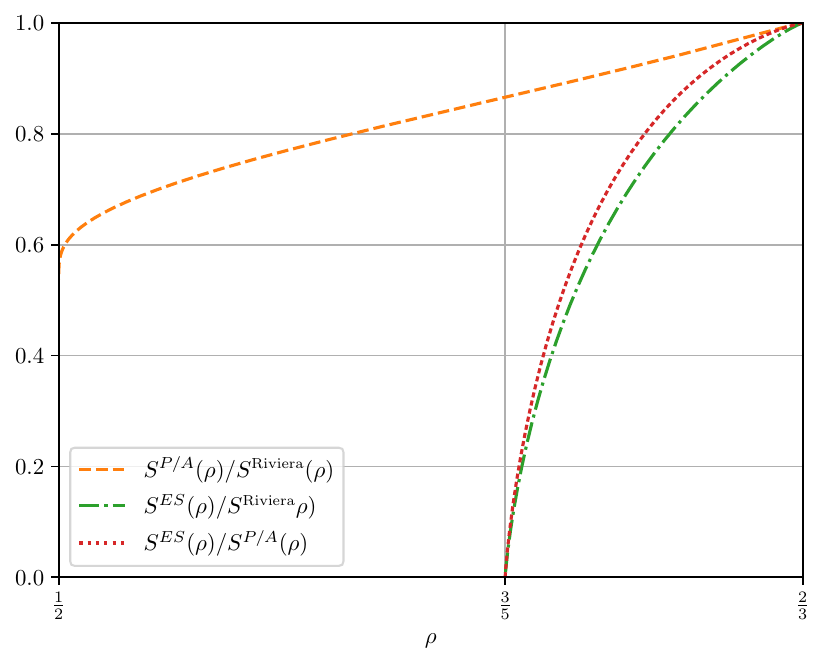}
		\caption{}
	\end{subfigure}\caption{The complexity function $S^{ES}(\rho)$ of evolutionary stable configurations compared to the complexity function of the Riviera model and the complexity function of the maximal configurations resistant to predators/altruists (denoted by $S^{P/A}(\rho)$ in the right subfigure).}\label{fig:Complexity_ALL}
\end{figure}

\begin{remark}
	The complexity function $S^{ES}(\rho)$ calculated in Theorem \ref{tm:cmpl_ES} is shown in Figure \ref{fig:Complexity_ALL} where it is compared to the complexity function of the full Riviera model, as computed in \cite{krapivsky2022jamming}, and the complexity function of the maximal configurations resistant to predators/altruists $S^{P}(\rho)=S^{A}(\rho)$, as computed in Theorems \ref{tm:cmpl_predatori} and \ref{tm:cmpl_altruisti}.
\end{remark}

\section{Two-dimensional case}\label{sec:general}
In this section, we make some observations related to maximal configurations resistant to predators, maximal configurations resistant to altruists, and evolutionary stable configurations, but in the two-dimensional case. Here, we are in the original setting in which the combinatorial settlement planning model was introduced (see \cite{PSZ-21, PSZ-21-2}). The main conclusions in this section that are related to maximal configurations resistant to predators, and to maximal configurations resistant to altruists, refer to the case when $m$ and $n$ grow to infinity. In the case of evolutionary stable configurations, we get much more precise results. Our main goal is to find $\rho_{\min}$ and $\rho_{\max}$, i.e.\ the minimal and the maximal density that can be achieved in each of the three situations. We already saw that, in one-dimensional case, maximal possible density coincides in all three situations and is equal to $\frac{2}{3}$ (which coincides with maximal reachable density in the original Riviera model). However, even though the minimal possible density in the case of maximal configurations resistant to predators, and maximal configurations resistant to altruists is the same as in the original Riviera model, namely $\frac{1}{2}$, in the case of evolutionary stable configurations, the minimal reachable density is $\frac{3}{5}$. We investigate what happens in the two-dimensional case. We show that only the maximal configurations resistant to altruists have the minimal and the maximal reachable densities that coincide with the minimal and the maximal reachable densities in the original combinatorial settlement planning model. These minimal and maximal possible densities in the original model were obtained in \cite{PSZ-21}. The minimal possible density is $\frac{1}{2}$, and the maximal possible density is $\frac{3}{4}$. When it comes to the maximal configurations resistant to predators, the density of $\frac{1}{2}$ can be achieved, but the highest possible density (in the limit) is $\frac{2}{3}$. This is somewhat surprising, because it seems (intuitively) that, if one wants to construct a maximal configuration resistant to predators, one should just have as many occupied lots as possible. It is a bit counter-intuitive that enforcing higher building density, spoils the resistance to arrival of predators. This intuition is based on thinking about tracts of land on the whole $\mathbb{Z}^2$ grid. However, it turns out that the boundary condition that we impose (that there is no obstruction to the sunlight on the boundary of the tract of land) has a very strong effect on the whole configuration. Lastly, we find that the minimal and the maximal reachable densities for evolutionary stable configurations, surprisingly, coincide and are equal to $\frac{2}{3}$. It turns out that the condition of resistance to both predators and altruists is so strong in the two-dimensional case, that we are able to completely describe all configurations that satisfy this condition.

\subsection{Maximal configurations resistant to predators}
In this subsection, we find the minimal and the maximal densities that can be achieved by the maximal configurations resistant to predators in the two-dimensional setting. Since maximal configurations resistant to predators are also maximal configurations in the original combinatorial settlement planning model, we know that the minimal reachable density is bigger than or equal to $\frac{1}{2}$, and the maximal reachable density is less than or equal to $\frac{3}{4}$. Since it turns out that the minimal density that can be reached with maximal configurations resistant to predators is precisely $\frac{1}{2}$, it is enough to identify some pattern which provides us with maximal configurations resistant to predators with density $\frac{1}{2}$. Many different patterns, that can appear in maximal configurations from the original combinatorial settlement planning model, were considered (together with their densities) in \cite{PSZ-21}, where the original model was introduced. The patterns were typically introduced on the whole $\mathbb{Z}^2$ grid, and then finite configurations obtained from those patterns were analyzed. It was very important to see how these infinite patterns restrict to finite size tracts of land, because on such finite configurations the boundary condition plays an important role. Recall that we assumed that the tracts of land are not adjacent to any other buildings, i.e.\ along the boundary of the rectangular $m \times n$ grid, there are no obstructions to sunlight. There were several patterns introduced in \cite{PSZ-21} which achieved the lowest possible density of $\frac{1}{2}$. One of these patterns was the so-called check pattern (see Figure \ref{fig:checkfigure}).
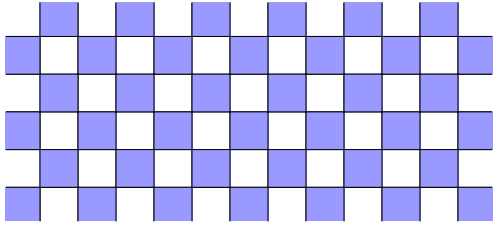
\begin{figure}
		\centering
		\begin{tikzpicture}[scale = 0.5]
		\draw[step=1cm,black,very thin] (-0.9, -0.9) grid (11.9,4.9);

		\fill[blue!40!white] (-0.9,-0.9) rectangle (0,0);
		\fill[blue!40!white] (-0.9,1) rectangle (0,2);
		\fill[blue!40!white] (-0.9,3) rectangle (0,4);

		\fill[blue!40!white] (0,0) rectangle (1,1);
		\fill[blue!40!white] (0,2) rectangle (1,3);
		\fill[blue!40!white] (0,4) rectangle (1,4.9);

		\fill[blue!40!white] (1,-0.9) rectangle (2,0);
		\fill[blue!40!white] (1,1) rectangle (2,2);
		\fill[blue!40!white] (1,3) rectangle (2,4);

		\fill[blue!40!white] (2,0) rectangle (3,1);
		\fill[blue!40!white] (2,2) rectangle (3,3);
		\fill[blue!40!white] (2,4) rectangle (3,4.9);

		\fill[blue!40!white] (3,-0.9) rectangle (4,0);
		\fill[blue!40!white] (3,1) rectangle (4,2);
		\fill[blue!40!white] (3,3) rectangle (4,4);

		\fill[blue!40!white] (4,0) rectangle (5,1);
		\fill[blue!40!white] (4,2) rectangle (5,3);
		\fill[blue!40!white] (4,4) rectangle (5,4.9);

		\fill[blue!40!white] (5,-0.9) rectangle (6,0);
		\fill[blue!40!white] (5,1) rectangle (6,2);
		\fill[blue!40!white] (5,3) rectangle (6,4);

		\fill[blue!40!white] (6,0) rectangle (7,1);
		\fill[blue!40!white] (6,2) rectangle (7,3);
		\fill[blue!40!white] (6,4) rectangle (7,4.9);

		\fill[blue!40!white] (7,-0.9) rectangle (8,0);
		\fill[blue!40!white] (7,1) rectangle (8,2);
		\fill[blue!40!white] (7,3) rectangle (8,4);

		\fill[blue!40!white] (8,0) rectangle (9,1);
		\fill[blue!40!white] (8,2) rectangle (9,3);
		\fill[blue!40!white] (8,4) rectangle (9,4.9);

		\fill[blue!40!white] (9,-0.9) rectangle (10,0);
		\fill[blue!40!white] (9,1) rectangle (10,2);
		\fill[blue!40!white] (9,3) rectangle (10,4);

		\fill[blue!40!white] (10,0) rectangle (11,1);
		\fill[blue!40!white] (10,2) rectangle (11,3);
		\fill[blue!40!white] (10,4) rectangle (11,4.9);

		\fill[blue!40!white] (11,-0.9) rectangle (11.9,0);
		\fill[blue!40!white] (11,1) rectangle (11.9,2);
		\fill[blue!40!white] (11,3) rectangle (11.9,4);

		\draw[step=1cm,black,very thin] (-0.9, -0.9) grid (11.9,4.9);
		\end{tikzpicture}
		\caption{The check pattern.}\label{fig:checkfigure}
	\end{figure}
Crucial property of the check pattern is that we can trivially obtain finite size configurations from it, by just adding houses on all the empty boundary lots (see Figure \ref{fig:checkpatternexamples}).
\begin{figure}
	\centering
		\begin{subfigure}{.5\textwidth}
			\centering
			\begin{tikzpicture}[scale = 0.5]
			\draw[step=1cm,black,very thin] (-1, 0) grid (10,4);

			\fill[blue!40!white] (-1,0) rectangle (0,1);
			\fill[blue!40!white] (-1,1) rectangle (0,2);
			\fill[blue!40!white] (-1,2) rectangle (0,3);
			\fill[blue!40!white] (-1,3) rectangle (0,4);

			\fill[blue!40!white] (0,0) rectangle (1,1);
			\fill[blue!40!white] (0,2) rectangle (1,3);

			\fill[blue!40!white] (1,0) rectangle (2,1);
			\fill[blue!40!white] (1,1) rectangle (2,2);
			\fill[blue!40!white] (1,3) rectangle (2,4);

			\fill[blue!40!white] (2,0) rectangle (3,1);
			\fill[blue!40!white] (2,2) rectangle (3,3);

			\fill[blue!40!white] (3,0) rectangle (4,1);
			\fill[blue!40!white] (3,1) rectangle (4,2);
			\fill[blue!40!white] (3,3) rectangle (4,4);

			\fill[blue!40!white] (4,0) rectangle (5,1);
			\fill[blue!40!white] (4,2) rectangle (5,3);

			\fill[blue!40!white] (5,0) rectangle (6,1);
			\fill[blue!40!white] (5,1) rectangle (6,2);
			\fill[blue!40!white] (5,3) rectangle (6,4);

			\fill[blue!40!white] (6,0) rectangle (7,1);
			\fill[blue!40!white] (6,2) rectangle (7,3);

			\fill[blue!40!white] (7,0) rectangle (8,1);
			\fill[blue!40!white] (7,1) rectangle (8,2);
			\fill[blue!40!white] (7,3) rectangle (8,4);

			\fill[blue!40!white] (8,0) rectangle (9,1);
			\fill[blue!40!white] (8,2) rectangle (9,3);

			\fill[blue!40!white] (9,0) rectangle (10,1);
			\fill[blue!40!white] (9,1) rectangle (10,2);
			\fill[blue!40!white] (9,2) rectangle (10,3);
			\fill[blue!40!white] (9,3) rectangle (10,4);

			\draw[step=1cm,black,very thin] (-1, 0) grid (10,4);
			\end{tikzpicture}
			\caption{}
		\end{subfigure}%
		\begin{subfigure}{.5\textwidth}
			\centering
			\begin{tikzpicture}[scale = 0.5]
			\draw[step=1cm,black,very thin] (-1, -1) grid (5,8);
			\fill[blue!40!white] (-1,-1) rectangle (0,0);
			\fill[blue!40!white] (-1,0) rectangle (0,1);
			\fill[blue!40!white] (-1,1) rectangle (0,2);
			\fill[blue!40!white] (-1,2) rectangle (0,3);
			\fill[blue!40!white] (-1,3) rectangle (0,4);
			\fill[blue!40!white] (-1,4) rectangle (0,5);
			\fill[blue!40!white] (-1,5) rectangle (0,6);
			\fill[blue!40!white] (-1,6) rectangle (0,7);
			\fill[blue!40!white] (-1,7) rectangle (0,8);

			\fill[blue!40!white] (0,-1) rectangle (1,0);
			\fill[blue!40!white] (0,0) rectangle (1,1);
			\fill[blue!40!white] (0,2) rectangle (1,3);
			\fill[blue!40!white] (0,4) rectangle (1,5);
			\fill[blue!40!white] (0,6) rectangle (1,7);

			\fill[blue!40!white] (1,-1) rectangle (2,0);
			\fill[blue!40!white] (1,1) rectangle (2,2);
			\fill[blue!40!white] (1,3) rectangle (2,4);
			\fill[blue!40!white] (1,5) rectangle (2,6);
			\fill[blue!40!white] (1,7) rectangle (2,8);

			\fill[blue!40!white] (2,-1) rectangle (3,0);
			\fill[blue!40!white] (2,0) rectangle (3,1);
			\fill[blue!40!white] (2,2) rectangle (3,3);
			\fill[blue!40!white] (2,4) rectangle (3,5);
			\fill[blue!40!white] (2,6) rectangle (3,7);

			\fill[blue!40!white] (3,-1) rectangle (4,0);
			\fill[blue!40!white] (3,1) rectangle (4,2);
			\fill[blue!40!white] (3,3) rectangle (4,4);
			\fill[blue!40!white] (3,5) rectangle (4,6);
			\fill[blue!40!white] (3,7) rectangle (4,8);

			\fill[blue!40!white] (4,-1) rectangle (5,0);
			\fill[blue!40!white] (4,0) rectangle (5,1);
			\fill[blue!40!white] (4,1) rectangle (5,2);
			\fill[blue!40!white] (4,2) rectangle (5,3);
			\fill[blue!40!white] (4,3) rectangle (5,4);
			\fill[blue!40!white] (4,4) rectangle (5,5);
			\fill[blue!40!white] (4,5) rectangle (5,6);
			\fill[blue!40!white] (4,6) rectangle (5,7);
			\fill[blue!40!white] (4,7) rectangle (5,8);

			\draw[step=1cm,black,very thin] (-1, -1) grid (5,8);
			\end{tikzpicture}
			\caption{}
		\end{subfigure}%
	\caption{Finite configurations obtained from the check pattern.}\label{fig:checkpatternexamples}
  \end{figure}
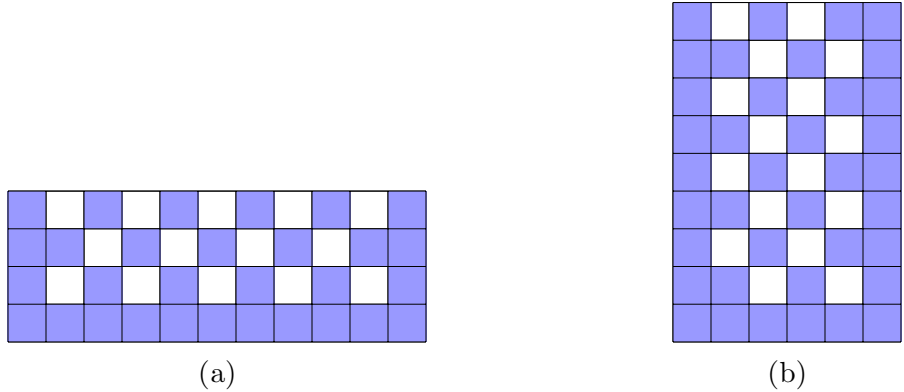
Clearly, the density achieved (in the limit) by the finite size configurations with check pattern is still $\frac{1}{2}$ since the effect of the boundary vanishes as the configuration grows. Furthermore, notice that each empty lot in these configurations has the property that there are houses on all the neighboring lots from which the sunlight can come (east, west, and south). Hence, a house built on any of the empty lots would not receive sunlight. This precisely means that such configurations are resistant to predators. Hence, we can achieve density $\frac{1}{2}$ with maximal configurations resistant to predators.

When it comes to maximal configurations resistant to predators with the highest possible density, things get much more complicated. Regarding patterns, the obvious candidate is the only pattern that exhibits the density of $\frac{3}{4}$, the so-called brick pattern (see Figure \ref{fig:brickfigure}).
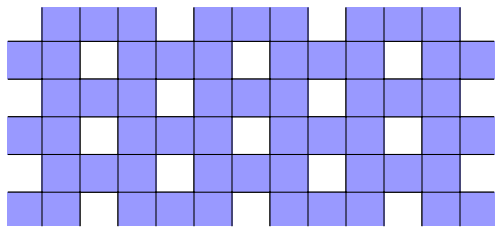
\begin{figure}
	\centering
		\begin{tikzpicture}[scale = 0.5]
		\draw[step=1cm,black, thin] (-0.9, -0.9) grid (11.9,4.9);
		\fill[blue!40!white] (-0.9,-0.9) rectangle (0,0);
		\fill[blue!40!white] (0,-0.9) rectangle (1,0);

		\fill[blue!40!white] (2,-0.9) rectangle (3,0);
		\fill[blue!40!white] (3,-0.9) rectangle (4,0);
		\fill[blue!40!white] (4,-0.9) rectangle (5,0);

		\fill[blue!40!white] (6,-0.9) rectangle (7,0);
		\fill[blue!40!white] (7,-0.9) rectangle (8,0);
		\fill[blue!40!white] (8,-0.9) rectangle (9,0);

		\fill[blue!40!white] (10,-0.9) rectangle (11,0);
		\fill[blue!40!white] (11,-0.9) rectangle (11.9,0);
		\fill[blue!40!white] (0,0) rectangle (1,1);
		\fill[blue!40!white] (1,0) rectangle (2,1);
		\fill[blue!40!white] (2,0) rectangle (3,1);

		\fill[blue!40!white] (4,0) rectangle (5,1);
		\fill[blue!40!white] (5,0) rectangle (6,1);
		\fill[blue!40!white] (6,0) rectangle (7,1);

		\fill[blue!40!white] (8,0) rectangle (9,1);
		\fill[blue!40!white] (9,0) rectangle (10,1);
            \fill[blue!40!white] (10,0) rectangle (11,1);
		\fill[blue!40!white] (-0.9,1) rectangle (0,2);
		\fill[blue!40!white] (0,1) rectangle (1,2);

		\fill[blue!40!white] (2,1) rectangle (3,2);
		\fill[blue!40!white] (3,1) rectangle (4,2);
		\fill[blue!40!white] (4,1) rectangle (5,2);

		\fill[blue!40!white] (6,1) rectangle (7,2);
		\fill[blue!40!white] (7,1) rectangle (8,2);
		\fill[blue!40!white] (8,1) rectangle (9,2);

		\fill[blue!40!white] (10,1) rectangle (11,2);
		\fill[blue!40!white] (11,1) rectangle (11.9,2);
		\fill[blue!40!white] (0,2) rectangle (1,3);
		\fill[blue!40!white] (1,2) rectangle (2,3);
		\fill[blue!40!white] (2,2) rectangle (3,3);

		\fill[blue!40!white] (4,2) rectangle (5,3);
		\fill[blue!40!white] (5,2) rectangle (6,3);
		\fill[blue!40!white] (6,2) rectangle (7,3);

		\fill[blue!40!white] (8,2) rectangle (9,3);
		\fill[blue!40!white] (9,2) rectangle (10,3);
            \fill[blue!40!white] (10,2) rectangle (11,3);
		\fill[blue!40!white] (-0.9,3) rectangle (0,4);
		\fill[blue!40!white] (0,3) rectangle (1,4);

		\fill[blue!40!white] (2,3) rectangle (3,4);
		\fill[blue!40!white] (3,3) rectangle (4,4);
		\fill[blue!40!white] (4,3) rectangle (5,4);

		\fill[blue!40!white] (6,3) rectangle (7,4);
		\fill[blue!40!white] (7,3) rectangle (8,4);
		\fill[blue!40!white] (8,3) rectangle (9,4);

		\fill[blue!40!white] (10,3) rectangle (11,4);
		\fill[blue!40!white] (11,3) rectangle (11.9,4);
		\fill[blue!40!white] (0,4) rectangle (1,4.9);
		\fill[blue!40!white] (1,4) rectangle (2,4.9);
		\fill[blue!40!white] (2,4) rectangle (3,4.9);

		\fill[blue!40!white] (4,4) rectangle (5,4.9);
		\fill[blue!40!white] (5,4) rectangle (6,4.9);
		\fill[blue!40!white] (6,4) rectangle (7,4.9);

		\fill[blue!40!white] (8,4) rectangle (9,4.9);
		\fill[blue!40!white] (9,4) rectangle (10,4.9);
		\fill[blue!40!white] (10,4) rectangle (11,4.9);
		\draw[step=1cm,black, thin] (-0.9, -0.9) grid (11.9,4.9);
		\end{tikzpicture}
    \caption{The brick pattern.}\label{fig:brickfigure}
\end{figure}
If we look at this pattern on the whole $\mathbb{Z}^2$, it looks like a perfect candidate for maximal configuration resistant to predators, since again each empty lot is blocked from sunlight from all sides. However, restricting this pattern to a finite size tract of land will necessarily produce some empty lots in the bottom row where predators could build a house (see Figure \ref{fig:brickpatternexamples}). If we tried exchanging positions of such an empty lot in the bottom row and the house above it, we would end up with two neighboring empty lots in the configuration, which would, again, make it possible for a predator to build a house.
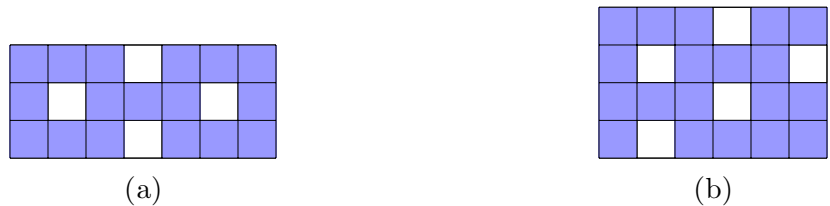
\begin{figure}
	\centering
		\begin{subfigure}{.5\textwidth}
			\centering
			\begin{tikzpicture}[scale = 0.5]
			\draw[step=1cm,black,very thin] (0, 0) grid (7,3);
			\fill[blue!40!white] (0,0) rectangle (1,1);
			\fill[blue!40!white] (0,1) rectangle (1,2);
			\fill[blue!40!white] (0,2) rectangle (1,3);

			\fill[blue!40!white] (1,0) rectangle (2,1);
			\fill[blue!40!white] (1,2) rectangle (2,3);

			\fill[blue!40!white] (2,1) rectangle (3,2);
			\fill[blue!40!white] (2,0) rectangle (3,1);
			\fill[blue!40!white] (2,2) rectangle (3,3);

			\fill[blue!40!white] (3,1) rectangle (4,2);

			\fill[blue!40!white] (4,0) rectangle (5,1);
			\fill[blue!40!white] (4,1) rectangle (5,2);
			\fill[blue!40!white] (4,2) rectangle (5,3);

			\fill[blue!40!white] (5,0) rectangle (6,1);
			\fill[blue!40!white] (5,2) rectangle (6,3);

			\fill[blue!40!white] (6,0) rectangle (7,1);
			\fill[blue!40!white] (6,1) rectangle (7,2);
			\fill[blue!40!white] (6,2) rectangle (7,3);

			\draw[step=1cm,black,very thin] (0, 0) grid (7,3);
			\end{tikzpicture}
			\caption{}\label{brickpattern1}
		\end{subfigure}%
		\begin{subfigure}{.5\textwidth}
			\centering
			\begin{tikzpicture}[scale = 0.5]
			\draw[step=1cm,black,very thin] (0, 0) grid (6,4);
			\fill[blue!40!white] (0,0) rectangle (1,1);
			\fill[blue!40!white] (0,1) rectangle (1,2);
			\fill[blue!40!white] (0,2) rectangle (1,3);
			\fill[blue!40!white] (0,3) rectangle (1,4);

			\fill[blue!40!white] (1,1) rectangle (2,2);
			\fill[blue!40!white] (1,3) rectangle (2,4);

			\fill[blue!40!white] (2,0) rectangle (3,1);
			\fill[blue!40!white] (2,1) rectangle (3,2);
			\fill[blue!40!white] (2,2) rectangle (3,3);
			\fill[blue!40!white] (2,3) rectangle (3,4);

			\fill[blue!40!white] (3,0) rectangle (4,1);
			\fill[blue!40!white] (3,2) rectangle (4,3);

			\fill[blue!40!white] (4,0) rectangle (5,1);
			\fill[blue!40!white] (4,1) rectangle (5,2);
			\fill[blue!40!white] (4,2) rectangle (5,3);
			\fill[blue!40!white] (4,3) rectangle (5,4);

			\fill[blue!40!white] (5,0) rectangle (6,1);
			\fill[blue!40!white] (5,1) rectangle (6,2);
			\fill[blue!40!white] (5,3) rectangle (6,4);

			\draw[step=1cm,black,very thin] (0, 0) grid (6,4);
			\end{tikzpicture}
			\caption{}\label{brickpattern2}
		\end{subfigure}
    \caption{Finite configurations obtained from the brick pattern.}\label{fig:brickpatternexamples}
\end{figure}
The fact that we cannot make a restriction of the brick patter resistant to predators is not without reason. As it turns out, no maximal configurations resistant to predators can (in the limit) achieve densities greater than $\frac{2}{3}$. This immediately follows from the following two lemmas.

\begin{lemma}
	The occupancy of the penultimate row of an $m\times n$ maximal configuration resistant to predators is at most $2\CEIL{\frac{n}{3}}$.
\end{lemma}
\begin{proof}
	First notice that all the lots on the western, southern and eastern border of the tract of land have to be occupied. Let us now inspect what can happen in the row immediately above the southern-most row. Since the southern-most row has all the lots occupied, we cannot have three (or more) consecutive occupied lots in the row above (this would clearly cause that at least one house in such a block is blocked from the sunlight). Hence, the highest possible occupancy that we can achieve in the penultimate row is $2\CEIL{\frac{n}{3}}$, which in the limit gives the density of $\frac{2}{3}$,
\end{proof}

Now we argue that the occupancy of the penultimate row essentially sets the upper bound for the occupancies of all the other rows above.

\begin{lemma}
	If the occupancy of the penultimate row of an $m\times n$  maximal configuration resistant to predators is $r$, then the occupancy of any other row above is at most $r+1$.
\end{lemma}
\begin{proof}
	Let us denote by $\ell$ the number of empty lots in any particular row. The claim of the lemma will follow if we can show two things: firstly, that the number of empty lots in the row directly above the considered row is at least $\ell - 1$ (so the number of empty lots increased, stayed the same, or decreased by $1$), and secondly, in case the number of empty lots decreased, then the number of empty lots in the row directly above that one must again increase and be at least $\ell$. From this it trivially follows that the occupancy of the penultimate row (increased by one) really provides the upper bound for the occupancy of all the other rows above it.

	To see that the number of empty lots can decrease by at most $1$ going from one row to the next above it, we make the following key observation. Since it is not allowed to have two neighboring empty lots anywhere in a maximal configuration resistant to predators, we know that between any two nearest empty lots in the same row, there is at least one occupied lot. Furthermore, the lots immediately to the north of those two empty lots, are certainly occupied, and between those two occupied lots, there is at least one empty lot to keep the configuration permissible. Hence, each two nearest empty lots in the same row produce at least one empty lot in the row above. This implies that if the number of empty lots in a particular row is equal to $\ell$, the number of empty lots in the row above is at least $\ell - 1$.

	It only remains to be seen what happens in the case when the number of empty lots actually decreases by $1$ going from one row to the next one above. To be specific, take $i$ to be the index of the considered row with $\ell$ empty lots, and let row $i-1$ be the row directly above it with $\ell-1$ empty lots. Note that in that case there is exactly one empty lot produced in row $i-1$ from each two nearest empty lots in row $i$. The column of that empty lot in row $i-1$ is in between the columns of the two empty lots in row $i$. More importantly, there cannot be empty lots in row $i-1$ to the left of the first empty lot in row $i$, nor to the right of the last empty lot in row $i$.

	Let us first analyze the initial portion of the row $i$. It cannot start with an empty lot since predator could build a house on that lot, and it cannot start with $4$ or more occupied lots, since then both second and third occupied lot would need to get the sunlight from the south, and this would create two neighboring empty lots, which would again make it possible for a predator to build a house there.
	If the first empty lot in row $i$ is in the third column, it will produce an additional empty lot to the left in row $i-1$, see \eqref{eq:1st_empty_on_3rd}.
	\begin{equation}\label{eq:1st_empty_on_3rd}
		\begin{aligned}
			& \qquad 1 \quad 0 \quad 1 \\
			\textnormal{row $i$} \longrightarrow & \qquad 1 \quad 1 \quad 0 \quad 1 \quad \cdots
		\end{aligned}
	\end{equation}
	If the first empty lot is in the fourth column in row $i$, it, again, will produce an additional empty lot to the left in row $i-1$, see \eqref{eq:1st_empty_on_4th}. The difference between this case, and the previous one where the first empty lot was in the third column, is that it is not uniquely determined where the imposed empty lot will be. It can be on any of the two positions marked with `$x$' in \eqref{eq:1st_empty_on_4th}.
	\begin{equation}\label{eq:1st_empty_on_4th}
		\begin{aligned}
			& \qquad 1 \quad x \quad x  \quad 1 \\
			\textnormal{row $i$} \longrightarrow & \qquad 1 \quad 1 \quad 1 \quad 0 \quad 1 \quad \cdots
		\end{aligned}
	\end{equation}
	Thus, the only possibility left is that the first empty lot in row $i$ must be in the second column, see \eqref{eq:1st_empty_on_2nd}.
	\begin{equation}\label{eq:1st_empty_on_2nd}
		\begin{aligned}
			& \qquad 1 \quad 1 \\
			\textnormal{row $i$} \longrightarrow & \qquad 1 \quad 0 \quad 1 \quad \cdots
		\end{aligned}
	\end{equation}

	Due to the symmetry in the horizontal (east--west) direction, we can similarly conclude that the last empty lot in row $i$ must be in the second to last column, and only then none of them will impose an additional empty lot in row $i-1$, see \eqref{eq:1st_empty_on_2nd_and_penultimate}.
	\begin{equation}\label{eq:1st_empty_on_2nd_and_penultimate}
		\begin{aligned}
			& \qquad 1 \quad 1 \quad \phantom{1} \quad \cdots \quad \phantom{1} \quad 1 \quad 1 \\
			\textnormal{row $i$} \longrightarrow & \qquad 1 \quad 0 \quad 1 \quad \cdots \quad 1 \quad 0 \quad 1
		\end{aligned}
	\end{equation}
	Inspecting \eqref{eq:1st_empty_on_2nd_and_penultimate}, we can easily see that this will put us in the situation where in row $i-1$, the first empty lot is not on the second position, and the last empty lot is not on the second to last position. But, reasoning similarly as before, this implies that in row $i-2$ there must be an additional empty lot both to the left of the first empty lot in row $i-1$ and to the right of the last empty lot in row $i-1$. This means that the number of empty lots in row $i-2$ is at least $2+(\ell-1)-1=\ell$. We therefore return to the same occupancy that we had two rows below. Hence, if the occupancy of the penultimate row is $r$, we clearly have that the occupancy of any other row above is bounded above by $r+1$ which is exactly what we needed to show.
\end{proof}

In conclusion --- since the penultimate row cannot have the density bigger than $\frac{2}{3}$ (in the limit), neither can the other rows above it. It, hence, follows that the maximal density achievable by maximal configurations resistant to predators is bounded above by $\frac{2}{3}$. In Subsection \ref{subsec:ES_2d} we discuss evolutionary stable maximal configurations which are, in particular, resistant to predators. There, we show that they all have the same density of $\frac{2}{3}$ (see Theorem \ref{tm:ES_2d_structure}). This means that the bound above is sharp and completes the argument that $\rho_{\max} = \frac{2}{3}$ in the case of maximal configurations resistant to predators.

\subsection{Maximal configurations resistant to altruists}
The main goal of this subsection is to show that maximal configurations resistant to altruists (in two-dimensional case) can achieve the smallest and the highest density from the original combinatorial settlement planning model. As in the first part of the previous subsection, the main idea is to identify appropriate patterns which achieve those densities. It turns out that, for both the lowest and the highest possible density, one can just directly take the patterns discussed and analyzed in details in \cite{PSZ-21}.

Let us first discuss the lowest possible density that maximal configurations resistant to altruists can achieve. In \cite[Theorem 4.5]{PSZ-21}, sharp lower bounds were provided for the occupancy of any maximal configuration $C$. More precisely, it was proven that if $C$ is any maximal configuration on the $m \times n$ grid (with $m, n \ge 2$), then
\begin{equation*}
    |C| \ge \begin{cases}
		\frac{mn}{2}+2, &\text{if } n\equiv 0\Mod{4},\\
		\frac{m(n+2)}{2}, &\text{if } n\equiv 2\Mod{4},\\
		\frac{m(n+1)}{2}+1, &\text{if } n\equiv 1\Mod{2},\\
	\end{cases}
\end{equation*}
where $|C|$ is the occupancy of the configuration $C$, i.e.\ the total number of occupied lots in the configuration $C$. To prove that these bounds are sharp, the authors in \cite{PSZ-21} found a precise pattern with which this particular occupancy is achieved, and this is the so-called rake-stripe pattern (see Figure \ref{fig:comb_rake_stripe}).
\begin{figure}
		\centering
		\begin{subfigure}{.5\textwidth}
			\centering
			\begin{tikzpicture}[scale = 0.5]
			\draw[step=1cm,black,very thin] (3, 0) grid (11,6);

			\fill[blue!40!white] (3,0) rectangle (4,1);

			\fill[blue!40!white] (4,0) rectangle (5,1);
			\fill[blue!40!white] (4,1) rectangle (5,2);
			\fill[blue!40!white] (4,2) rectangle (5,3);
			\fill[blue!40!white] (4,3) rectangle (5,4);
			\fill[blue!40!white] (4,4) rectangle (5,5);
			\fill[blue!40!white] (4,5) rectangle (5,6);

			\fill[blue!40!white] (5,0) rectangle (6,1);
			\fill[blue!40!white] (5,1) rectangle (6,2);
			\fill[blue!40!white] (5,2) rectangle (6,3);
			\fill[blue!40!white] (5,3) rectangle (6,4);
			\fill[blue!40!white] (5,4) rectangle (6,5);
			\fill[blue!40!white] (5,5) rectangle (6,6);

			\fill[blue!40!white] (6,0) rectangle (7,1);
			\fill[blue!40!white] (7,0) rectangle (8,1);

			\fill[blue!40!white] (8,0) rectangle (9,1);
			\fill[blue!40!white] (8,1) rectangle (9,2);
			\fill[blue!40!white] (8,2) rectangle (9,3);
			\fill[blue!40!white] (8,3) rectangle (9,4);
			\fill[blue!40!white] (8,4) rectangle (9,5);
			\fill[blue!40!white] (8,5) rectangle (9,6);

			\fill[blue!40!white] (9,0) rectangle (10,1);
			\fill[blue!40!white] (9,1) rectangle (10,2);
			\fill[blue!40!white] (9,2) rectangle (10,3);
			\fill[blue!40!white] (9,3) rectangle (10,4);
			\fill[blue!40!white] (9,4) rectangle (10,5);
			\fill[blue!40!white] (9,5) rectangle (10,6);

			\fill[blue!40!white] (10,0) rectangle (11,1);
			\draw[step=1cm,black,very thin] (3, 0) grid (11,6);
			\end{tikzpicture}
			\caption{The rake pattern only.}
		\end{subfigure}%
		\begin{subfigure}{.5\textwidth}
			\centering
			\begin{tikzpicture}[scale = 0.5]
			\draw[step=1cm,black,very thin] (2, 0) grid (10,6);
			\fill[blue!40!white] (2,0) rectangle (3,1);
			\fill[blue!40!white] (2,1) rectangle (3,2);
			\fill[blue!40!white] (2,2) rectangle (3,3);
			\fill[blue!40!white] (2,3) rectangle (3,4);
			\fill[blue!40!white] (2,4) rectangle (3,5);
			\fill[blue!40!white] (2,5) rectangle (3,6);

			\fill[blue!40!white] (3,1) rectangle (4,2);
			\fill[blue!40!white] (3,3) rectangle (4,4);
			\fill[blue!40!white] (3,5) rectangle (4,6);

			\fill[blue!40!white] (4,1) rectangle (5,2);
			\fill[blue!40!white] (4,3) rectangle (5,4);
			\fill[blue!40!white] (4,5) rectangle (5,6);

			\fill[blue!40!white] (5,1) rectangle (6,2);
			\fill[blue!40!white] (5,3) rectangle (6,4);
			\fill[blue!40!white] (5,5) rectangle (6,6);

			\fill[blue!40!white] (6,1) rectangle (7,2);
			\fill[blue!40!white] (6,3) rectangle (7,4);
			\fill[blue!40!white] (6,5) rectangle (7,6);

			\fill[blue!40!white] (7,1) rectangle (8,2);
			\fill[blue!40!white] (7,3) rectangle (8,4);
			\fill[blue!40!white] (7,5) rectangle (8,6);

			\fill[blue!40!white] (8,1) rectangle (9,2);
			\fill[blue!40!white] (8,3) rectangle (9,4);
			\fill[blue!40!white] (8,5) rectangle (9,6);

			\fill[blue!40!white] (9,0) rectangle (10,1);
			\fill[blue!40!white] (9,1) rectangle (10,2);
			\fill[blue!40!white] (9,2) rectangle (10,3);
			\fill[blue!40!white] (9,3) rectangle (10,4);
			\fill[blue!40!white] (9,4) rectangle (10,5);
			\fill[blue!40!white] (9,5) rectangle (10,6);

			\draw[step=1cm,black,very thin] (2, 0) grid (10,6);
			\end{tikzpicture}
			\caption{The stripe pattern only.}
		\end{subfigure}
		\begin{subfigure}{.5\textwidth}
			\centering
			\begin{tikzpicture}[scale = 0.5]
			\draw[step=1cm,black,very thin] (3, 0) grid (11,6);

			\fill[blue!40!white] (3,0) rectangle (4,1);
			\fill[blue!40!white] (3,1) rectangle (4,2);

			\fill[blue!40!white] (4,1) rectangle (5,2);
			\fill[blue!40!white] (4,2) rectangle (5,3);
			\fill[blue!40!white] (4,3) rectangle (5,4);
			\fill[blue!40!white] (4,4) rectangle (5,5);
			\fill[blue!40!white] (4,5) rectangle (5,6);

			\fill[blue!40!white] (5,1) rectangle (6,2);
			\fill[blue!40!white] (5,2) rectangle (6,3);
			\fill[blue!40!white] (5,3) rectangle (6,4);
			\fill[blue!40!white] (5,4) rectangle (6,5);
			\fill[blue!40!white] (5,5) rectangle (6,6);

			\fill[blue!40!white] (6,1) rectangle (7,2);
			\fill[blue!40!white] (7,1) rectangle (8,2);

			\fill[blue!40!white] (8,1) rectangle (9,2);
			\fill[blue!40!white] (8,2) rectangle (9,3);
			\fill[blue!40!white] (8,3) rectangle (9,4);
			\fill[blue!40!white] (8,4) rectangle (9,5);
			\fill[blue!40!white] (8,5) rectangle (9,6);

			\fill[blue!40!white] (9,1) rectangle (10,2);
			\fill[blue!40!white] (9,2) rectangle (10,3);
			\fill[blue!40!white] (9,3) rectangle (10,4);
			\fill[blue!40!white] (9,4) rectangle (10,5);
			\fill[blue!40!white] (9,5) rectangle (10,6);

			\fill[blue!40!white] (10,0) rectangle (11,1);
			\fill[blue!40!white] (10,1) rectangle (11,2);
			\draw[step=1cm,black,very thin] (3, 0) grid (11,6);
			\end{tikzpicture}
			\caption{The combination of the two patterns.}
		\end{subfigure}
	\caption{Examples of the rake pattern, the stripe pattern, and the combination of the two, on the tract of land with dimensions $6 \times 8$.}\label{fig:comb_rake_stripe}
	\end{figure}
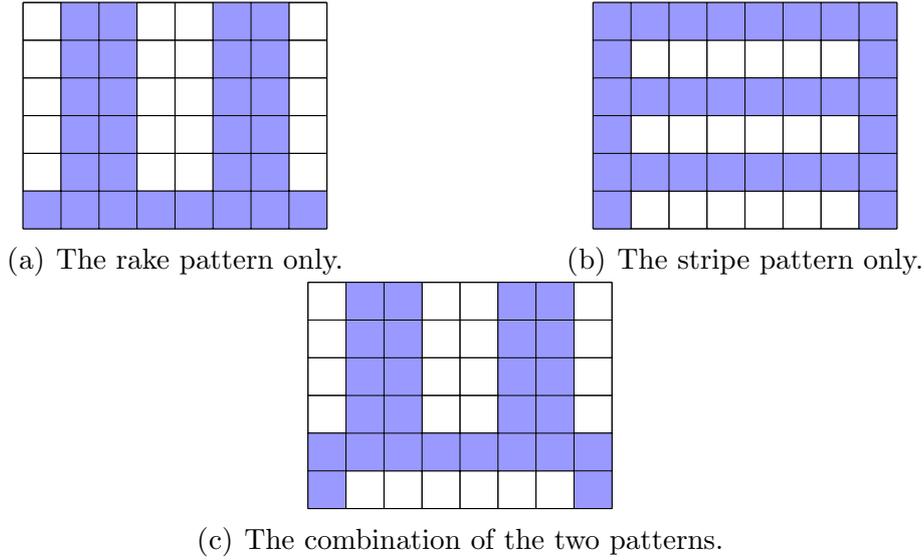
It is easy to see that each of the three patterns displayed in Figure \ref{fig:comb_rake_stripe} produce density $\frac{1}{2}$ in the limit, and it is also obvious that each empty lot in those three patterns has a property that it is the only source of light to one of its neighbors. This means that all three patterns are resistant to altruists. Hence, not only can we achieve the density $\frac{1}{2}$ with maximal configurations resistant to altruists, but even the maximal configurations with theoretically the lowest possible occupancy in the combinatorial settlement model introduced in \cite{PSZ-21} are resistant to altruists.

When it comes to the highest possible density achievable by maximal configurations resistant to altruists, we can just take maximal configurations composed of the so-called brick pattern. It is true that finite configurations obtained from the brick pattern can exhibit empty lots on the boundary, which is a problem when allowing predators. However, those configurations still have the crucial property that every empty lot is the only source of light to at least one of its neighbors (see Figure \ref{fig:brickpatternexamples}), which directly implies that such configurations are resistant to altruists. Notice that the brick pattern is composed in such a way that every empty lot is the only source of light to each of its neighbors --- to the east, west and north (see Figure \ref{fig:brickfigure}). Since the restriction of the brick pattern on a finite size tract of land, starts with the western-most column completely occupied, empty boundary lots can only appear on the southern, and the eastern boundary. Those on the southern boundary will still be the only source of light to its neighbor to the north, and those on the eastern boundary will still be the only source of light to its western neighbor.

\subsection{Evolutionary stable configurations}\label{subsec:ES_2d}
Notice that in the case when the number of rows ($m$) is equal to $2$, the two-dimensional model reduces to the model on a strip of land that we already analyzed in Section \ref{sec:ES}. More precisely, all the lots in the bottom row are necessarily occupied (to counteract the arrival of predators), and this immediately blocks the sunlight from the south to all the lots in the upper row. Hence, the situation in the upper row entirely corresponds to the situation in the one-dimensional version of the model. With the case $m\le 2$ completely understood, we now move on to the case $m>2$. When $n\le 2$ the only maximal configuration is the one with all lots occupied and that configuration is also ES. For $m,n>2$ the situation is much more interesting. It turns out that there are no ES configurations on an $m\times n$ grid unless $n$ is divisible by $3$ and $m$ is an odd number. If these two conditions are met then we can completely describe all the ES configurations and they all have the same occupancy of $\frac23 mn + \frac12 (m-1) + \frac13 n$, see Theorem \ref{tm:ES_2d_structure}, which, in the limit, results with the density of $\frac{2}{3}$.

\begin{remark}\label{rem:ES_bricked_up_border}
	Note that each $m\times n$ ES configuration is completely bricked up along its eastern, western and southern border, i.e.\ $C_{i,j}=1$ whenever $j\in\{1,n\}$ or $i=m$. Otherwise, a predator would gladly take up any empty lot along those borders.
\end{remark}

\begin{remark}\label{rem:no_side_adjecent_empty}
	Note that each empty lot in any ES configuration must be surrounded on all four sides by occupied lots. Equivalently, there are no (side-)adjacent empty lots in ES configurations, as a predator could take up one of them.
\end{remark}

Before we are in position to prove Theorem \ref{tm:ES_2d_structure}, describing the structure of two-dimensional ES configurations, we will need to establish a number of preparatory lemmas.
\begin{lemma}\label{lm:finite_forbiden_I_II}
	If $C$ is an $m\times n$ ES configuration, then a finite constellation of either of the types shown in Figure \ref{fig:forbidden_constellations_I_II} must not appear anywhere in $C$.
	\begin{figure}
		\begin{subfigure}{\textwidth}
			$$\begin{array}{cccc|ccc|c|ccc|ccc}
				0 & * & * & * & * & * & * & \dots & * & * & * & * & * & * \\
				* & \underline{0} & * & * & * & * & * & \dots & * & * & * & 1 & 1 & 1 \\
				* & * & 0 & 1 & 1 & 0 & 1 & \dots & 1 & 0 & 1 & 1 & 0 & 1 \\
				& & & \multicolumn{9}{c}{\underbrace{\hspace{10em}}_{k\text{-times}}} & &
			\end{array}$$
			\caption{Type I constellation.}
		\end{subfigure}
		\begin{subfigure}{\textwidth}
			$$\begin{array}{cccc|ccc|c|ccc|ccc}
				* & \underline{0} & * & * & * & * & * & \dots & * & * & * & 1 & 1 & 1 \\
				0 & * & 0 & 1 & 1 & 0 & 1 & \dots & 1 & 0 & 1 & 1 & 0 & 1 \\
				& & & \multicolumn{9}{c}{\underbrace{\hspace{10em}}_{k\text{-times}}} & &
			\end{array}$$
			\caption{Type II constellation.}
		\end{subfigure}
		\caption{Forbidden constellations of Type I and II. The inner length $3$ blocks may not be present, or can be repeated any finite number of times in either of the configuration types ($k\ge 0$).}\label{fig:forbidden_constellations_I_II}
	\end{figure}
\end{lemma}
\begin{proof}
	We argue by contradiction. Let us assume that such a constellation appears in $C$, and let $(i,j)$ denote the position in $C$ which matches the underlined zero in such a constellation. We may further assume that the constellation under consideration is such that the number $k$ of repetitions of the inner length $3$ block is the least possible. Note that this means that the constellation under consideration is necessarily of Type I as having a constellation of Type II implies that, locally, the configuration must look like this
	$$\begin{array}{cccc|ccc|c|ccc|ccc}
		1 & 1 & 1 & 0 & * & * & * & \dots & * & * & * & * & * & * \\
		1 & \underline{0} & 1 & 1 & \mathbf{0} & * & * & \dots & * & * & * & 1 & 1 & 1 \\
		0 & 1 & 0 & 1 & 1 & 0 & 1 & \dots & 1 & 0 & 1 & 1 & 0 & 1
	\end{array}$$
	which would contradict the minimality of $k$, as the boldface zero marks the beginning of the Type I configuration which has one repetition of the inner block less than the initial constellation.

	Let us now consider what the surroundings of the constellation of Type I must look like. Clearly $C_{i-1,j}=C_{i,j+1}=1$. Next, $C_{i-1,j+1}=0$ as $C_{i-1,j+1}=1$ would, otherwise, imply $C_{i-1,j+2}=C_{i,j+3}=0$ which would, again, contradict the minimality of $k$.

	Next, $C_{i,j+2}=1$, as otherwise $C_{i,j+2}=0$, together with resistance to altruists, would imply the existence of the row $i-2$ with $C_{i-2,j}=C_{i-2,j+1}=C_{i-2,j+2}=1$ which would, along with $C_{i-2,j-1}=1$, contradict permissibility.

	Next, $C_{i,j+3}=0$ and $C_{i-1,j+2}=C_{i-1,j+3}=C_{i,j+4}=1$. As $C_{i-1,j+4}=1$ would imply $C_{i-1,j+5}=C_{i,j+6}=C_{i+1,j+7}=0$ which would, again contradict the minimality of $k$, we are forced to conclude $C_{i-1,j+4}=0$, and hence $C_{i-1,j+5}=1$. The conclusions so far produce the following constellation:
	$$\begin{array}{cccc|ccc|c|ccc|ccc}
		0 & 1 & 0 & 1 & 1 & 0 & 1 & \dots & * & * & * & * & * & * \\
		* & \underline{0} & 1 & 1 & 0 & 1 & * & \dots & * & * & * & 1 & 1 & 1 \\
		* & * & 0 & 1 & 1 & 0 & 1 & \dots & 1 & 0 & 1 & 1 & 0 & 1
	\end{array}$$

	Now if $C_{i,j+5}=1$, the reasoning from the previous paragraph can be repeated substituting $j$ for $j+3$. We inductively conclude that, locally, $C$ must look like as follows:
	$$\begin{array}{cccc|ccc|c|ccc|ccc|c|ccc}
		0 & 1 & 0 & 1 & 1 & 0 & 1 & \dots & 1 & 0 & 1 & 1 & 0 & 1 & \dots & * & * & * \\
		* & \underline{0} & 1 & 1 & 0 & 1 & 1 & \dots & 0 & 1 & 1 & 0 & 1 & 0 & \dots & 1 & 1 & 1 \\
		* & * & 0 & 1 & 1 & 0 & 1 & \dots & 1 & 0 & 1 & 1 & 0 & 1 & \dots & 1 & 0 & 1\\
		& & & \multicolumn{9}{c}{\underbrace{\hspace{10em}}_{l\text{-times}}} & \multicolumn{2}{c}{}
	\end{array}$$
	Note that we must, at some point, have $C_{i,j+3(l+1)+2}=0$ as the rightmost block of length $3$ implies that $C_{i,n-3}=0$.

	Next, since the $C$ is resistant to altruists, the row $i-2$ must exist, and necessarily $C_{i-2,j+3(l+1)}=C_{i-2,j+3(l+1)+1}=C_{i-2,j+3(l+1)+2}=1$ and, clearly, $C_{i-2,j-1}=C_{i-2,j+1}=1$ and $C_{i-2,j}=0$.
	$$\begin{array}{cccc|ccc|c|ccc|ccc|c|ccc}
		1 & \mathbf{0} & 1 & * & * & * & * & \dots & * & * & * & 1 & 1 & 1 & \dots & * & * & * \\
		0 & 1 & 0 & 1 & 1 & 0 & 1 & \dots & 1 & 0 & 1 & 1 & 0 & 1 & \dots & * & * & * \\
		* & \underline{0} & 1 & 1 & 0 & 1 & 1 & \dots & 0 & 1 & 1 & 0 & 1 & 0 & \dots & 1 & 1 & 1 \\
		* & * & 0 & 1 & 1 & 0 & 1 & \dots & 1 & 0 & 1 & 1 & 0 & 1 & \dots & 1 & 0 & 1\\
		& & & \multicolumn{9}{c}{\underbrace{\hspace{10em}}_{l\text{-times}}} & \multicolumn{2}{c}{}
	\end{array}$$
	But this, again, leads us to a contradiction with minimality of $k$ as the boldface zero marks the beginning of the constellation of Type II which is clearly shorter than the initial constellation. As all the possibilities have been exhausted, this concludes the proof of the lemma.
\end{proof}

\begin{lemma}\label{lm:forbiden_III_IV}
	If $C$ is an $m\times n$ ES configuration, then neither of the two constellations shown in Figure \ref{fig:forbidden_constellations_III_IV} is to appear anywhere in $C$.
	\begin{figure}
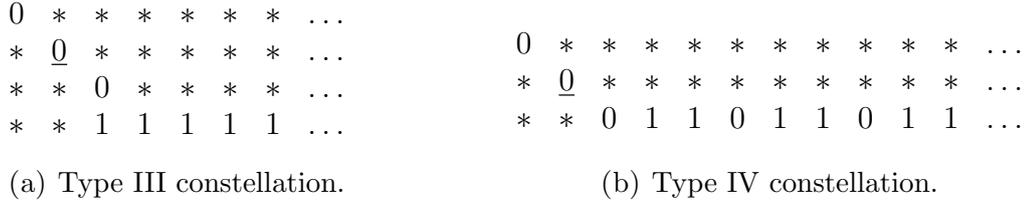

		\centering
		\begin{subfigure}{.48\textwidth}
			$$\begin{array}{cccccccc}
				0 & * & * & * & * & * & * & \dots \\
				* & \underline{0} & * & * & * & * & * & \dots \\
				* & * & 0 & * & * & * & * & \dots \\
				* & * & 1 & 1 & 1 & 1 & 1 & \dots
			\end{array}$$
		\caption{Type III constellation.}
		\end{subfigure}
		\hfill
		\begin{subfigure}{.48\textwidth}
			$$\begin{array}{cccccccccccc}
				0 & * & * & * & * & * & * & * & * & * & * & \dots \\
				* & \underline{0} & * & * & * & * & * & * & * & * & * & \dots \\
				* & * & 0 & 1 & 1 & 0 & 1 & 1 & 0 & 1 & 1 & \dots
			\end{array}$$
			\caption{Type IV constellation.}
		\end{subfigure}
	\caption{Forbidden constellations of Type III and IV. It is assumed that in both constellations the pattern in the bottom row repeats indefinitely (with period $1$ for Type III or period $3$ for Type IV) all the way to the right border of the configuration.}\label{fig:forbidden_constellations_III_IV}
	\end{figure}
\end{lemma}
\begin{proof}
	We argue by contradiction. Let us assume that either of those constellations appears in $C$, and let $(i,j)$ denote the position in $C$ which matches the underlined zero of the said constellation. Thus $C_{i-1,j-1}=C_{i,j}=C_{i+1,j+1}=0$.

	We may further assume that $j$ is chosen the largest possible, i.e.\ this is the rightmost position where such a constellation (either of Type III or IV) appears in $C$.

	The following conclusions are immediate $C_{i-1,j}=C_{i,j+1}=1$. Also $C_{i-1,j+1}=0$ as  $C_{i-1,j+1}=1$ would otherwise imply $C_{i-1,j+2}=C_{i,j+3}=C_{i+1,j+4}=0$ and this would contradict the maximality of $j$.

	Next, $C_{i-1,j+2}=C_{i+1,j+2}=1$. Since $C$ is resistant to altruists, either $C_{i-2,j}=C_{i-2,j+1}=C_{i-2,j+2}=1$, or $C_{i,j+2}=C_{i-1,j+3}=1$. We later argue what happens in the former case. In the latter case we get $C_{i,j+3}=0$, then $C_{i,j+4}=C_{i+1,j+3}=1$ and $C_{i+1,j+4}=0$. Now we find ourselves in a similar situation as before. It must be that $C_{i-1,j+4}=0$ as $C_{i-1,j+4}=1$ would otherwise imply $C_{i-1,j+5}=C_{i,j+6}=C_{i+1,j+7}=0$ and this would contradict the maximality of $j$.

	Next, $C_{i-1,j+5}=C_{i+1,j+5}=1$. Since $C$ is resistant to altruists, either $C_{i-2,j+3}=C_{i-2,j+4}=C_{i-2,j+5}=1$, or $C_{i,j+5}=C_{i-1,j+6}=1$. We later argue what happens in the former case. In the latter we can inductively repeat the argument.

	At some point this has to stop as the configuration is finite, and then we end up in the `former case' branch of the argument, and we are forced to conclude that the row $i-2$ above our constellation exists and the portion of $C$ to the right of the position $(i,j)$ must be of the following form:
	$$\begin{array}{cccc|ccc|ccc|c|ccc}
		1 & \mathbf{0} & 1 & * & *&*&* & *&*&* & \quad\dots\quad~ & 1&1&1 \\
		0 & 1 & 0 & 1 & 1&0&1 & 1&0&1 & \quad\dots\quad~ & 1&0&1 \\
		* & \underline{0} & 1 & 1 & 0&1&1 & 0&1&1 & \quad\dots\quad~ & 0&1&* \\
		* & * & 0 & 1 & 1&0&1 & 1&0&1 & \quad\dots\quad~ & 1&0&1
	\end{array}$$
	But this is a contradiction as the boldface zero marks the beginning of the constellation of Type II which is forbidden by Lemma \ref{lm:finite_forbiden_I_II}. This concludes the proof of the lemma.
\end{proof}

\begin{remark}\label{rm:symmetry}
	Both Lemma \ref{lm:finite_forbiden_I_II} and Lemma \ref{lm:forbiden_III_IV} have mirror versions obtained by reflecting all the statements and arguments horizontally via East$\leftrightarrow$West involution.
\end{remark}

\begin{lemma}\label{lm:begining_penultimate_row}
	If $C$ is an $m\times n$ ES configuration, where $m,n>2$, then $C_{m-1,2}=C_{m-1,n-1}=0$.
\end{lemma}
\begin{proof}
	Let us first assume, conversely, that $C_{m-1,2}=1$. By Remark \ref{rem:ES_bricked_up_border} we know that that the lower left corner of $C$ looks like
	$$\begin{array}{ccccc}
		1 & * & * & * & \dots \\
		1 & 1 & * & * & \dots \\
		1 & 1 & 1 & 1 & \dots
	\end{array}$$
	Next, $C_{m-1,3}=0$, then $C_{m-1,4}=C_{m-2,3}=1$ and $C_{m-2,2}=0$. Resistance to altruists next implies existence of the row $m-3$ and $C_{m-3,1}=C_{m-3,2}=C_{m-3,3}=1$. This finally implies $C_{m-3,4}=C_{m-2,5}=C_{m-1,6}=0$ and the configuration is as follows:
	$$\begin{array}{cccccccc}
		1 & 1 & 1 & 0 & * & * & * & \dots \\
		1 & 0 & 1 & 1 & \mathbf{0} & * & * & \dots \\
		1 & 1 & 0 & 1 & 1 & 0 & * & \dots \\
		1 & 1 & 1 & 1 & 1 & 1 & 1 & \dots
	\end{array}$$
	Note that the boldface zero marks the beginning of the constellation of Type III which is forbidden by Lemma \ref{lm:forbiden_III_IV}. This contradiction means that $C_{m-1,2}=0$ after all.

	By the symmetry argument (Remark \ref{rm:symmetry}) we can similarly argue that $C_{m-1,n-1}=0$, thus completing the proof.
\end{proof}

\begin{lemma}\label{lm:no_one_apart}
	Let $C$ be an $m\times n$ ES configuration, where $m>2$. No two empty lots in the row $m-1$ (the second row from the bottom) are separated by a single occupied lot.
\end{lemma}
\begin{proof}
	Let us assume this did happen around the position $(m-1,j)$ (underlined in the diagram below), i.e.\ $C_{m-1,j-1}=C_{m-1,j+1}=0$ and $C_{m-1,j}=1$. Since $m>2$ we conclude $C_{m-2,j-1}=C_{m-2,j+1}=1$ and $C_{m-2,j}=0$. The resistance of $C$ to altruists implies $m>3$ and $C_{m-3,j-1}=C_{m-3,j}=C_{m-3,j+1}=1$. This finally leads to $C_{m-3,j+2}=C_{m-2,j+3}=C_{m-1,j+4}=0$ as in the diagram below:
	$$\begin{array}{ccccccccc}
		\dots & 1 & 1 & 1 & 0 & * & * & * & \dots \\
		\dots & 1 & 0 & 1 & 1 & \mathbf{0} & * & * & \dots \\
		\dots & 0 & \underline{1} & 0 & 1 & 1 & 0 & * & \dots \\
		\dots & 1 & 1 & 1 & 1 & 1 & 1 & 1 & \dots
	\end{array}$$
	The boldface zero above marks the beginning of the constellation of Type III which is forbidden by Lemma \ref{lm:forbiden_III_IV}. A contradiction.
\end{proof}

A direct consequence of Lemma \ref{lm:begining_penultimate_row} and Lemma \ref{lm:no_one_apart}, after taking into account that no three occupied lots are allowed in the penultimate row (row $m-1$), is the following proposition.
\begin{proposition}\label{prop:bottom_two_rows}
	Let $C$ be an $m\times n$ ES configuration, where $m,n>2$. Then $n$ is divisible by $3$ and the bottom two rows are necessarily of the form
	$$\begin{array}{ccc|ccc|c|ccc}
	1 & 0 & 1 & 1 & 0 & 1 & \dots & 1 & 0 & 1 \\
	1 & 1 & 1 & 1 & 1 & 1 & \dots & 1 & 1 & 1
	\end{array}$$
\end{proposition}

\begin{lemma}\label{lm:forbiden_V}
	If $C$ is an $m\times n$ ES configuration, then the constellation shown in Figure \ref{fig:forbidden_constellation_V} must not appear anywhere in $C$.
	\begin{figure}
		$$\begin{array}{ccccccccccc}
			\dots & * & * & * & * & \underline{1} & * & * & * & * & \dots \\
			\dots & * & * & * & * & * & * & * & * & * & \dots \\
			\dots & 1 & 0 & 1 & 1 & 0 & 1 & 1 & 0 & 1 & \dots
		\end{array}$$
		\caption{Forbidden constellation of Type V. It is assumed that the pattern in the bottom row repeats indefinitely (with period $3$) all the way from the left to the right border of the configuration.}\label{fig:forbidden_constellation_V}
	\end{figure}
\end{lemma}
\begin{proof}
	We argue by contradiction. Let us assume that the constellation of Type V does appear in $C$, and let $(i,j)$ denote the position in $C$ which matches the underlined $1$ in the said constellation. Clearly, $C_{i+1,j}=1$. It is forbidden that both $C_{i,j-1}=C_{i,j+1}=1$, so let us assume, without loss of generality that $C_{i,j+1}=0$. Next, $C_{i+1,j+1}=1$ and $C_{i+1,j+2}=0$. The configuration $C$ is forced to look like:
	$$\begin{array}{cccccccccccccc}
		\dots & * & * & * & * & \underline{1} & 0 & * & * & * & * & * & * & \dots \\
		\dots & * & * & * & * & 1 & 1 & \mathbf{0} & * & * & * & * & * & \dots \\
		\dots & 1 & 0 & 1 & 1 & 0 & 1 & 1 & 0 & 1 & 1 & 0 & 1 & \dots
	\end{array}$$
	Note that the boldface zero above marks the beginning of Type IV constellation which is forbidden by Lemma \ref{lm:forbiden_III_IV}. A contradiction.
\end{proof}

\begin{proposition}\label{prop:odd_rows_from_the bottom}
	Let $C$ be an $m\times n$ ES configuration, where $m,n>2$. Then all the rows $m-1-2k$, for $k\ge0$ (which are present) look exactly the same as the row $m-1$:
	$$101\,101\,101\,\dots\,101$$
\end{proposition}
\begin{proof}
	We argue inductively on $k$. For $k=0$, the statement was proven in Proposition \ref{prop:bottom_two_rows}. Let us now assume that for some $k\ge 0$ the row $m-1-2k$ is
	$$101\,101\,101\,\dots\,101$$
	and let us consider the row $m-1-2k-2$ (assuming it is present).
	By Lemma \ref{lm:forbiden_V} this row must match all the empty lots from the row $m-1-2k$ and is therefore
	$$1\,0*\,\,*\,0*\,\,*\,0*\,\dots\,*0\,1$$
	but by Remark \ref{rem:no_side_adjecent_empty} it immediately follows that the row $m-1-2k-2$ also looks like
	$$101\,101\,101\,\dots\,101$$
	which completes the inductive step.
\end{proof}

The following theorem now completely characterizes all the $m\times n$ ES configurations for $m,n>2$.

\begin{theorem}\label{tm:ES_2d_structure}
	Let $C$ be an $m\times n$ ES configuration for $m,n>2$. Then $n$ is divisible by $3$ and $m$ is odd, and all the ES configurations have the same structure as shown in the constellation bellow:
	\begin{equation}\label{eq:ES_2d_structure}
	\begin{array}{cccccccccccccccc}
		1 & 1 & * & * & 1 & * & * & 1 & * & \dots & * & 1 & * & * & 1 & 1 \\
		1 & 0 & 1 & 1 & 0 & 1 & 1 & 0 & 1 & \dots & 1 & 0 & 1 & 1 & 0 & 1 \\
		\vdots & \vdots & \vdots & \vdots & \vdots & \vdots & \vdots & \vdots & \vdots & \vdots & \vdots & \vdots & \vdots & \vdots & \vdots & \vdots \\
		1 & 0 & 1 & 1 & 0 & 1 & 1 & 0 & 1 & \dots & 1 & 0 & 1 & 1 & 0 & 1 \\
		1 & 1 & * & * & 1 & * & * & 1 & * & \dots & * & 1 & * & * & 1 & 1 \\
		1 & 0 & 1 & 1 & 0 & 1 & 1 & 0 & 1 & \dots & 1 & 0 & 1 & 1 & 0 & 1 \\
		1 & 1 & 1 & 1 & 1 & 1 & 1 & 1 & 1 & \dots & 1 & 1 & 1 & 1 & 1 & 1
	\end{array}
	\end{equation}
	where exactly one in each pair of the adjacent $*$ lots is occupied, and the other is empty. As a consequence, all the ES configurations have the same occupancy of $mn-\frac{(m-1)(2n-3)}{6} = \frac{2}{3}mn+\frac{1}{2}(m-1)+\frac{1}{3}n$.
\end{theorem}
\begin{proof}
	The structure of the ES configuration follows immediately from Proposition \ref{prop:odd_rows_from_the bottom}. The only thing remained to argue is that $m$, the number of rows, must be odd. Otherwise, the top two rows would be:
	$$\begin{array}{cccccccccccccccc}
	1 & 0 & 1 & 1 & 0 & 1 & 1 & 0 & 1 & \dots & 1 & 0 & 1 & 1 & 0 & 1 \\
	1 & 1 & * & * & 1 & * & * & 1 & * & \dots & * & 1 & * & * & 1 & 1
	\end{array}$$
	where exactly one in each pair of the adjacent $*$ lots is occupied, and the other is empty. Let us encode these choices by a string of letters $L$ and $R$ depending whether the empty lot of the pair is the left or the right one. Note that resistance to altruists implies $C_{2,3}=C_{2,n-2}=1$ so the sequence starts with $R$ and ends with $L$. Consequently, there exists a position in the string $R\dots L$ where the letters $R$ and $L$ are consecutive, and in that order. The corresponding part of the configuration would then have to look like this:
	$$\begin{array}{cccccccccccccccc}
		\dots & 1 & 0 & 1 & 1 & \mathbf{0} & 1 & 1 & 0 & 1 & \dots \\
		\dots & 1 & 1 & 1 & 0 & 1 & 0 & 1 & 1 & 1 & \dots
	\end{array}$$
	But this a contradiction as an altruist would gladly take up the empty lot indicated by the boldface zero above.
\end{proof}

Now that we completely described all the two dimensional ES configurations, it is only natural to ask how many are there. Naively, one might think that there are $2^{\frac{(m-1)(n-3)}{6}}$ of them, as it seems that for each of $\frac{(m-1)(n-3)}{6}$ adjacent pairs of $*$'s, we have to choose which lot is empty and which is occupied. But the truth turns out to be a bit more complicated than that.

\begin{theorem}
	The number of ES configurations on the $m\times n$ grid, when $m,n>2$, $3\mid n$ and $2\mid (m-1)$ is the same as the number of $M\times N$ grids, where $M=\frac{m-1}{2}$ and $N=\frac{n-3}{3}$, formed of letters $L$ and $R$ where the forbidden constellations are:
	$$\begin{array}{cc}
		R & *\\
		R & L
	\end{array}\quad, \quad \begin{array}{cc}
		* & L\\
		R & L
	\end{array}\quad, \quad \begin{array}{cc}
		| & L\\
		| & L
	\end{array}\quad, \quad \begin{array}{cc}
		R & |\\
		R & |
	\end{array}\quad .$$
	Each of the first two constellations stands for two different forbidden patterns, while the third forbids adjacent $L$'s in first column (along the left, west, border) and the fourth forbids adjacent $R$'s in last column (along the right, east, border).
\end{theorem}
\begin{remark}
	Obtaining a closed formula for the number of these two-dimensional $L/R$-configurations (if one even exists) seems to be prohibitively difficult. It is, perhaps, worth mentioning here that it is known to be computationally undecidable whether arbitrarily large $M\times N$ configurations even exist, given a prescribed set of forbidden patterns \cite{Berger66,Robinson71}.
\end{remark}
\begin{proof}
	Let us first consider the translations of the forbidden $L/R$-constellations above into $0/1$-constellations of occupied and empty lots:
	\begin{gather*}
		\begin{array}{ccccccc}
			1&1&0&1&*&*&1\\
			0&1&1&\underline{0}&1&1&0\\
			1&1&0&1&0&1&1\\
			0&1&1&0&1&1&0\\
		\end{array}\quad, \quad \begin{array}{ccccccc}
			1&*&*&1&0&1&1\\
			0&1&1&\underline{0}&1&1&0\\
			1&1&0&1&0&1&1\\
			0&1&1&0&1&1&0\\
		\end{array}\quad,\\
		\begin{array}{cccccc}
			|&1&1&0&1&1\\
			|&1&\underline{0}&1&1&0\\
			|&1&1&0&1&1\\
			|&1&0&1&1&0\\
		\end{array}\quad, \quad \begin{array}{cccccc}
			1&1&0&1&1&|\\
			0&1&1&\underline{0}&1&|\\
			1&1&0&1&1&|\\
			0&1&1&0&1&|\\
		\end{array}\quad .
	\end{gather*}
	It is not hard to check that the indicated empty lot in each of the four constellations marks the spot which invite an altruist to occupy it. This is why they are forbidden in an ES configuration.

	We now show that any configuration with the structure specified in Theorem \ref{tm:ES_2d_structure}, in which none of the $0/1$-constellations above appears, is indeed an ES configuration. Clearly, any such configuration is maximal (and permissible). It is not hard to see that it must, also, be resistant to predators.

	The empty lots in the odd indexed rows (one of the two adjacent $*$'s) of such a configuration are not inviting to altruists, as they represent the only source of light to the occupied lot represented by the other $*$. The two possibilities being:
	$$\begin{array}{cccc}
		1&\underline{0}&\underline{1}&1\\
		0&1&1&0\\
	\end{array} \quad\text{ or }\quad \begin{array}{cccc}
		1&\underline{1}&\underline{0}&1\\
		0&1&1&0\\
	\end{array}$$
	where the indicated lots are corresponding to $*$'s.

	Next, the interior empty lots in the even indexed rows (indicated in the constellation below):
	$$\begin{array}{ccccc}
		*&*&1&*&*\\
		1&1&\underline{0}&1&1\\
		*&*&1&*&*\\
	\end{array}$$
	are not inviting to altruists either, as the only way they could be is if the situation is actually
	$$\begin{array}{ccccc}
		1&0&1&*&*\\
		1&1&\underline{0}&1&1\\
		1&0&1&0&1\\
	\end{array} \quad\text{ or }\quad
	\begin{array}{ccccc}
		*&*&1&0&1\\
		1&1&\underline{0}&1&1\\
		1&0&1&0&1\\
	\end{array}$$
	which is forbidden by the two first patterns.

	We can reason similarly for the boundary empty lots in the even indexed rows (indicated in the constellations below):
	$$\begin{array}{ccccc}
		|&1&1&*&*\\
		|&1&\underline{0}&1&1\\
		|&1&1&*&*\\
	\end{array} \quad\text{ and }\quad \begin{array}{ccccc}
		*&*&1&1&|\\
		1&1&\underline{0}&1&|\\
		*&*&1&1&|\\
	\end{array}\quad.$$
	They would be inviting to altruists only if the actual situation were
		$$\begin{array}{ccccc}
		|&1&1&0&1\\
		|&1&\underline{0}&1&1\\
		|&1&1&0&1\\
	\end{array} \quad\text{ and }\quad \begin{array}{ccccc}
		1&0&1&1&|\\
		1&1&\underline{0}&1&|\\
		1&0&1&1&|\\
	\end{array}$$
	both of which are forbidden by the last two patterns.
\end{proof}

\section{Concluding remarks}\label{sec:concluding}

In this paper, we borrow a few concepts from the evolutionary game theory
to model (some) effects of a breakdown of social order under the
assumption of full rationality of agents. In particular, we consider a toy
model of urbanistic development known as the Riviera model and investigate
its robustness against two types of degradation by asocial actors.
One type of actor is characterized by predatory behavior, willing to harm
others, but not themselves; the other type is, in a sense, altruistic,
unwilling to harm others, but ready to accept sub-standard outcomes for
themselves. Invasions of either type are detrimental to law-abiding and
rules-respecting actors, who build communities cooperating within a given
set of rules. The main objects of our study are configurations that
have naturally evolved (following the rules) to a jammed (or a saturated)
state in which no further legal building is possible. We are interested
in characterizing and enumerating configurations that do not
require any external enforcement to resist illegal invasions, simply by
being unattractive to potential invaders, while still acceptable
to legal and rules-respecting owners. We call such configurations
evolutionary stable, indicating thus their robustness against both
considered types of invasions by non-cooperative players.
Our findings are quantified via the complexity functions of the maximal
configurations resistant to either type of degradation and to both of them.
Besides the complexity functions, we also find the bivariate generating
functions for the sequences enumerating all three types of resistant
configurations and determine their asymptotic behavior. In the end, we
revisit the original full-dimensional model and determine the main structural
properties of its configurations resistant to both types of degradation.

Our results pave the way toward modeling similar problems in more realistic
settings. They could have potential applications also in the context of
designing solar power plants, wireless energy transfer, and in all other
situations where unobstructed direct access to a spatially fixed resource
is crucial. A natural next step would be to search for, characterize, and
enumerate robust configurations on other types of lattices and in their
higher-dimensional analogons. The model could also be refined by introducing
additional rules, and by introducing various numerical parameters, leading
thus to more realistic results.

\section*{Acknowledgments}
\noindent
Partial support of the Slovenian ARIS (program P1-0383, grant no. J1-3002)
is gratefully acknowledged by T. Do\v{s}li\'c.

\bibliographystyle{bababbrv-fl}
\bibliography{literature}

\end{document}